\documentclass[10pt]{amsart}

\usepackage{amsmath, amssymb, amsthm}
\usepackage{rotating,xspace}
\usepackage{pst-all}

\numberwithin{equation}{section}

\theoremstyle{plain}
\newtheorem{prop}{Proposition}[section]
\newtheorem{coro}[prop]{Corollary}

\newtheorem{lemm}[prop]{Lemma}
\newtheorem{ques}[prop]{Question}

\theoremstyle{defi}
\newtheorem{defi}[prop]{Definition}
\newtheorem{algo}[prop]{Algorithm}
\newtheorem{nota}[prop]{Notation}
\newtheorem{exam}[prop]{Example}

\def\Reff#1; #2; #3; #4; #5; #6; #7\par{%
\bibitem{#1} #2, {\it #3}, #4 {\bf #5} (#6) #7}
\def\Ref#1; #2; #3; #4\par{%
\bibitem{#1} #2, {\it #3}, #4}

\psset{unit=1mm}
\psset{arrowsize=3pt}
\newpsstyle{exist}{linestyle=dashed,dash=2mm1mm}
\newpsstyle{thin}{linewidth=0.5pt}
\newpsstyle{thinexist}{linewidth=0.5pt, linestyle=dashed,dash=1.5mm 0.5mm}
\def\WPoint(#1,#2,#3){\cnode[style=thin,fillcolor=white,fillstyle=solid](#1,#2){0.5}{#3}}
\newcommand\WDot[1]{\pscircle[linewidth=.8pt, fillstyle=solid](#1){2pt}}
\def\ThickSegment(#1,#2){\ncline[linewidth=1.5pt,border=1.5pt,linecolor=black, nodesep=0pt]{#1}{#2}}
\definecolor{PcolorAA}{rgb}{0.8,0.75,1}
\definecolor{PcolorAAA}{rgb}{0.9,0.85,1}
\definecolor{PcolorBBB}{rgb}{1,.88,.85}

\renewcommand\aa{a}
\renewcommand\AA{A}
\def\act{\mathbin{\scriptscriptstyle\bullet}}
\newcommand\add{\mathsf{add}}

\newcommand\Ass{\mathrm{A}}
\newcommand\Augm{\mathsf{Aug}(\ZZZZ)}
\newcommand\Aut{\mathsf{Aut}}

\newcommand\bb{b}
\newcommand\BB{B}
\newcommand\Bi{B_\infty}
\newcommand\BRsp{B_\infty^{\mathsf{sp}}}
\newcommand\BWi{B\HS{-0.1}W_{\HS{-0.4}\infty}}

\newcommand\card{\mathtt{\#}}
\newcommand\cc{c}
\newcommand\chit{\widetilde{\chi}}
\newcommand\coIm{\mathsf{coIm}}
\newcommand\comp{\mathbin{\scriptscriptstyle\circ}}
\newcommand\concat{\mathord{{}^{\frown}}}
\newcommand\Conj{\mathsf{Conj}}
\newcommand\Conjj{\mathsf{Conj}_{\opR}}
\newcommand\Conjjj{\mathsf{Conj}_{\opR,\opRb}}
\newcommand\cut{\mathsf{cut}}
\newcommand\Cycl{\mathsf{Cycl}}

\newcommand\DISC{\hbox{\scshape{disc}}}
\renewcommand\div{\mathbin{\scriptstyle\sqsubset}}

\newcommand\divLD{\mathrel{\div_{\scriptscriptstyle\mathsf{LD}}}}
\newcommand\divsLD{\mathrel{\div^*_{\scriptscriptstyle\mathsf{LD}}}}
\newcommand\divs{\mathbin{\scriptstyle\sqsubset^*}}

\newcommand\ee{e}
\newcommand\EE{E}

\newcommand\eqL{\mathrel{=_{\scriptscriptstyle\LLL}}}
\newcommand\eqLD{\mathrel{=_{\scriptscriptstyle\mathsf{LD}}}\nobreak}
\newcommand\eqLDI{\mathrel{=_{\scriptscriptstyle\mathsf{LDI}}}\nobreak}
\newcommand\eqQuandle{\mathrel{=_{\scriptscriptstyle\mathsf{quandle}}}}
\newcommand\eqRack{\mathrel{=_{\scriptscriptstyle\mathsf{rack}}}}
\newcommand\EQUIV{\hbox{\scshape{equiv}}_{\Bi}}
\newcommand\eval{\mathsf{eval}}
\newcommand\EVAL{\hbox{\scshape{eval}}}
\newcommand\ew{\varepsilon}
\newcommand\expLD{\mathrel{\to_{\scriptscriptstyle\LD}}}

\newcommand\EXPAND{\hbox{\scshape{LD-expand}}}

\newcommand\False{\mathtt{false}}
\newcommand\ff{f}
\newcommand\FF{F}

\newcommand\FG[1]{F_{#1}}
\newcommand\Fi{F_\infty}
\newcommand\Found{\hbox{\scshape{found}}}
\newcommand\Free{\mathsf{Free}}
\renewcommand\ge{\geqslant\nobreak}
\renewcommand\gg{g}
\newcommand\GG{G}

\newcommand\Geom[1]{\mathsf{Geom}_{#1}}
\newcommand\Geomt[1]{\mathsf{G}\widetilde{\smash{\mathsf{eom}}}_{#1}}
\newcommand\Geomtp[1]{\mathsf{G}\widetilde{\smash{\mathsf{eom}}}^+_{#1}}

\newcommand\HalfConj{\mathsf{HalfConj}}
\newcommand\HS[1]{\leavevmode\null\hspace{#1mm}}

\newcommand\ie{{\it i.e.}}
\newcommand\ii{i}
\renewcommand\Im{\mathsf{Im}}
\newcommand\Inj[1]{\mathfrak{I}_{\HS{-0.2}{#1}}}
\newcommand\inv{^{-1}}
\newcommand\INV{\hbox{\scshape{inv}}}
\newcounter{ITEM}
\newcommand\ITEM[1]{\setcounter{ITEM}{#1}\leavevmode\hbox{\rm(\roman{ITEM})}}
\newcommand\Iter{\mathsf{Iter}}

\newcommand\jj{j}

\newcommand\kk{k}

\newcommand\lab[1]{\noindent\hbox to 5mm{\null\hfill\llap{\tiny#1:}\hspace{1mm}}}
\newcommand\LD{\mathrm{LD}}
\newcommand\LDI{\mathrm{LDI}}
\newcommand\LDop[1]{\hbox{\scshape{ld}}_{#1}}
\renewcommand\le{\leqslant\nobreak}
\newcommand\LL{L}
\newcommand\LLL{\mathcal{L}}
\newcommand\LS{\mathsf{left}}

\newcommand\mm{m}
\newcommand\mult{\mathbin{\scriptstyle\sqsupset}\nobreak}

\newcommand\mults{\mathbin{\scriptstyle\sqsupset^*}\nobreak}
\newcommand\multsLD{\mathrel{\mult^*_{\scriptscriptstyle\mathsf{LD}}}\nobreak}

\newcommand\nn{n}
\newcommand\NN{N}
\newcommand\NNNN{\mathbb{N}}
\newcommand\NT{\mathcal{N}}

\newcommand\op{\mathbin{\triangleright}\nobreak}

\newcommand\opL{\mathbin{\triangleright}}
\newcommand\opp{\mathord{\bullet}}
\newcommand\opR{\mathbin{\triangleleft}}
\newcommand\opRb{\mathbin{\overline\triangleleft}}

\newcommand\pdots{\mathrel{\HS{0.2}{\cdot}{\cdot}{\cdot}\HS{0.2}}}
\newcommand\Pol[1]{[#1]}
\newcommand\pp{p}

\newcommand\RD{\mathrm{RD}}

\newcommand\Red{\hbox{\scshape{red}}}
\newcommand\resp{\mbox{\it resp}.\ }
\newcommand\Rel[1]{\mathsf{Rel}_{#1}}
\newcommand\rr{r}
\newcommand\RR{R}

\newcommand\sh{\mathsf{sh}}
\newcommand\Shift{\mathsf{Shift}}
\newcommand\SHIFT{\hbox{\scshape{shift}}}
\newcommand\sig[1]{\sigma_{#1}}

\newcommand\sigg[2]{\sigma_{#1}^{#2}}
\newcommand\siginv[1]{\sigma_{#1}^{-1}}
\newcommand\SOL{\hbox{\scshape{sol}}}
\renewcommand\ss{s}
\renewcommand\SS{S}
\newcommand\SSS{\mathcal{S}}
\newcommand\sub[2]{#1_{\!/#2}}
\newcommand\Sym[1]{\mathfrak{S}_{\HS{-0.4}#1}}

\newcommand\TERM[2]{\mathsf{Term}_{#1}(#2)}
\newcommand\True{\mathtt{true}}
\newcommand\TT{T}
\newcommand\ttt{t}

\def\VR(#1,#2){\vrule width0pt height#1mm depth#2mm}
\newcommand\vs{{\it vs.}\ }

\newcommand\VV{V}

\newcommand\wdots{, ..., }
\newcommand\ww{w}

\newcommand\xx{x}
\newcommand\XX{X}

\newcommand\yy{y}

\newcommand\zz{z}

\newcommand\ZZZZ{\mathbb{Z}}

\author{Patrick DEHORNOY}
\address{Laboratoire de Math\'ematiques Nicolas Oresme UMR 6139\\ Universit\'e de Caen, 14032~Caen, France}
\email{patrick.dehornoy@unicaen.fr}
\urladdr{dehornoy.users.lmno.cnrs.fr}

\title{Some aspects of the SD-world}



\begin{document}

\begin{abstract}
We survey a few of the many results now known about the selfdistributivity law and selfdistributive structures, with a special emphasis on the associated word problems and the algorithms solving them in good cases.
\end{abstract}

\maketitle

Selfdistributivity (SD) is the algebraic law stating that a binary operation is distributive with respect to itself. It comes in two versions: \emph{left} self-distributivity $\xx(\yy\zz) = (\xx\yy)(\xx\zz)$, also called~LD, and \emph{right} self-distributivity $(\xx\yy)\zz = (\xx\zz)(\yy\zz)$, also called~RD. Their properties are, of course, entirely symmetric; however, because both versions occur in important examples (see below), it seems difficult to definitely choose one of them and stick to it. To avoid ambiguity, it is better to use an oriented operation symbol: we shall follow the excellent convention, now widely adopted, of using~$\opL$ for LD and $\opR$ for RD, both coherent with the intuition that the operation is an action on the term to which the triangle points. Then SD means that the action is compatible with the operation, and it takes the forms
\begin{gather}\label{LD}\tag{LD}
\xx \opL (\yy \opL \zz) = (\xx \opL \yy) \opL (\xx \opL \zz),\\
\label{RD}\tag{RD}
(\xx \opR \yy) \opR \zz = (\xx \opR \zz) \opR (\yy \opR \zz).
\end{gather}
Although selfdistributivity syntactically resembles associativity---only one letter separates $\xx(\yy\zz) = (\xx\yy)(\xx\zz)$ from $\xx(\yy\zz) = (\xx\yy)\zz$---, their properties are quite different, selfdistributivity turning out to be much more complicated: even the solution of the word problem and the description of free structures is highly nontrivial.

Selfdistributivity has been considered explicitly as early as the end of the XIXth century~\cite{Pei}, and selfdistributive structures have been extensively investigated, specially around Belousov in Kichinev from the late 1950s \cite{Bel1, Bel2} and around Je\v zek, Kepka, and N\v emec in Prague from the 1970s, with a number of structural results, in particular about two-sided selfdistributivity, see~\cite{JKN}. The interest in selfdistributivity was renewed and reinforced in the 1980s by the discovery of connections with low-dimensional topology by D.\,Joyce~\cite{Joy} and S.\,Matveev~\cite{Mat}, and with the theory of large cardinals in set theory by R.\,Laver~\cite{Lva} and the current author~\cite{Dem}. More recently, the connection with topology was made more striking by the cohomological approach proposed by R.\,Fenn, D.\,Rourke, and B.\,Sanderson~\cite{FRS}, considerably developed from the 1990s in work by J.S.\,Carter, S.\,Kamada, and other authors, see~\cite{CarterSurvey, CJKLS, CKS, FeR}; the study of quandles has now become a full subject, with intertwined papers both on the topological and on the algebraic side.

The aim of this text is to survey some aspects of selfdistributive algebra, with a special emphasis on the involved word problems. A comprehensive reference is the monograph~\cite{Dgd}, which is not always easy to read. We hope that this text will be more reader-friendly and might inspire further research (several open questions are mentioned).

The paper is organized in four sections. In the first section, we recall the now standard terminology and mention a few examples of selfdistributive structures, some classical, some more exotic. In Sections~\ref{S:WPLD1} and~\ref{S:WPLD2}, we survey the many results involving the word problem in the case of selfdistributivity alone. Finally, in Section~\ref{S:WPOther}, we similarly address the (easy) cases of racks and quandles and the (frustrating) case of spindles.

\section{Shelves, spindles, racks, and quandles}

Selfdistributive structures abound, and we begin with a few pictures from the SD-world. After recalling the usual terminology (Section~\ref{SS:Term}), we mention the classical examples (Section~\ref{SS:Classical}), and some more exotic ones (Section~\ref{SS:Exotic}). The description is summarized in a sort of chart of the SD-world (Section~\ref{SS:Chart}).

\subsection{Terminology}\label{SS:Term}

Throughout the paper, we use the now well established terminology introduced in topology.

\begin{defi}
A \emph{shelf} (or \emph{right-shelf}) is a structure~$(\SS, \opR)$, where $\opR$ is a binary operation on a (non-empty) set~$\SS$ that obeys the right selfdistributivity law~$\RD$. Symmetrically, a \emph{left-shelf} is a structure~$(\SS, \opL)$, with~$\opL$ obeying the left selfdistributivity law~$\LD$. 
\end{defi}

\begin{defi}
A \emph{spindle} is a shelf~$(\SS, \opR)$, in which the operation~$\opR$ obeys the idempotency law $\xx \opR \xx = \xx$.
\end{defi}

\begin{defi}
A \emph{rack} is a shelf~$(\SS, \opR)$, in which right translations are bijective, \ie, for every~$\bb$ in~$\SS$, the map $\RR_\bb: \xx \mapsto \xx \opR \bb$ is a bijection from~$\SS$ to~itself.
\end{defi}

Of course, \emph{left-racks} are left-shelves~$(\SS, \opL)$ with bijective left translations.

Racks can equivalently be defined as structures involving two operations:

\begin{prop}\label{P:InvOp}
If $(\SS, \opR)$ is a rack, and, for~$\bb, \cc$ in~$\SS$, we let $\cc \opRb \bb$ be the (unique) element~$\aa$ satisfying $\aa \opR \bb = \cc$, then $\opRb$ also obeys~RD, and $\opR$ and~$\opRb$ are connected by the mixed laws
\begin{equation}\label{E:InvOp}
(\xx \opR \yy) \opRb \yy = (\xx \opRb \yy) \opR \yy = \xx.
\end{equation}
Conversely, if $(\SS, \opR)$ is a shelf and there exists~$\opRb$ satisfying~\eqref{E:InvOp}, then $(\SS, \opR)$ is a rack and $\opRb$ is the second operation associated with~$\opR$ as above.
\end{prop}

The proof is an easy verification. Note that, in the situation of Prop.~\ref{P:InvOp}, each of the operations~$\opR, \opRb$ is distributive with respect to the other: \eqref{E:InvOp} implies 
\begin{gather*}
((\xx \opR \yy) \opRb \zz) \opR \zz = \xx \opR \yy, \text{ and }\\
((\xx \opRb \zz) \opR (\yy \opRb \zz)) \opR \zz = (((\xx \opRb \zz) \opR \zz) \opR ((\yy \opRb \zz)) \opR \zz) = \xx \opR \yy,
\end{gather*}
whence $(\xx \opR \yy) \opRb \zz = (\xx \opRb \zz) \opR (\yy \opRb \zz)$, whenever right translations of~$\opR$ are injective.

\begin{defi}
A \emph{quandle} is a rack~$(\SS, \opR)$, which is also a spindle, \ie, the operation~$\opR$ obeys the idempotency law $\xx \opR \xx = \xx$.
\end{defi}

As the examples below show, a rack need not be a quandle; however, a rack is always rather close to a quandle, in that the actions of an element and its square coincide: every rack satisfies the law $\xx \opR \yy = \xx \opR (\yy \opR \yy)$, as shows the computation
\begin{equation*}
\xx \opR (\yy \opR \yy) = ((\xx \opRb \yy) \opR \yy) \opR (\yy \opR \yy) = ((\xx \opRb \yy) \opR \yy) \opR \yy = \xx \opR \yy.
\end{equation*}

\subsection{Classical examples}\label{SS:Classical}

\begin{exam}[\bf trivial shelves]
Let $\SS$ be any set, and $\ff$ be any map from~$\SS$ to itself. Then defining $\aa \opR_\ff \bb := \ff(\aa)$ provides a selfdistributive operation on~$\SS$. The shelf $(\SS, \opR_\ff)$ is a spindle if, and only if, $\ff$ is the identity map; it is a rack if, and only if, $\ff$ is a bijection. Special cases are the {\bf cyclic racks}~$\Cycl_\nn$ corresponding to $\SS:= \ZZZZ{/}\nn\ZZZZ$ with $\aa \opR \bb := \aa + 1$, and the {\bf augmentation rack}~$\Augm$, corresponding to $\SS:= \ZZZZ$ with $\aa \opR \bb:= \aa + 1$ again.
\end{exam}

\begin{exam}[\bf lattice spindles]
If $(\LL, \wedge, \vee, 0, 1)$ is a lattice, then both $(\LL, \wedge)$ and $(\LL, \vee)$ are spindles. Moreover, they are both right- and left-spindles in the obvious sense, since the operations are commutative. Whenever $\LL$ has at least two elements, these spindles are not racks: for every~$\aa$ in~$\LL$, we have $0 \vee \aa = \aa \vee \aa = \aa$, so the right translation associated with a non-zero element is never injective.
\end{exam}

\begin{exam}[\bf Boolean shelves]
Let $(\BB, \wedge, \vee, 0, 1, \bar\  \ )$ be a Boolean algebra (\ie, a lattice that is distributive and complemented).  For~$\aa, \bb$ in~$\BB$, define $\aa \opR \bb:= \aa \vee \bar\bb$. Then $(\BB, \opR)$ is a shelf, as we find
\begin{align*}
(\xx \opR \yy) \opR \zz &= \xx \vee \bar\yy \vee \bar\zz = (\xx \vee \bar\yy \vee \bar\zz) \wedge 1 = ((\xx \vee \bar\zz) \vee \bar\yy) \wedge ((\xx \vee \bar\zz) \vee \zz)\\
&= (\xx \vee \bar\zz) \vee (\bar\yy \wedge \zz) = (\xx \vee \bar\zz) \vee (\overline{\yy \vee \bar\zz}) = (\xx \opR \zz) \opR (\yy \opR \zz).
\end{align*}
This shelf is neither a spindle, nor a rack for $\card\BB \ge 2$ as, for $\aa \not= 1$, we have $\aa \opR \aa = \aa \vee \bar\aa = 1$, and $\aa \opR \aa = 1 = 1 \opR \aa$. Note that, under the standard logical interpretation, $\opR$ corresponds to a reverse implication~$\Leftarrow$.
\end{exam}

\begin{exam}[\bf Alexander spindles]
Let $\RR$ be a ring and $\ttt$ belong to~$\RR$. Consider an $\RR$-module~$\EE$. For~$\aa, \bb$ in~$\EE$, define $\aa \opR \bb:= \ttt\aa + (1 - \ttt)\bb$. Then $(\EE, \opR)$ is a spindle, as $\opR$ is idempotent and we find
$$(\xx \opR \yy) \opR \zz = \ttt^2\xx + (\ttt - \ttt^2)\yy + (1 - \ttt)\zz = (\xx \opR \zz) \opR (\yy \opR \zz).$$
This spindle is a quandle if, and only if, $\ttt$ is invertible in~$\RR$, with the second operation then defined by $\cc \opRb \bb:= \ttt\inv \cc + (1 - \ttt\inv)\bb$.
\end{exam}

\begin{exam}[\bf conjugacy quandles]\label{X:Conj}
Let $\GG$ be a group. Define 
\begin{equation}\label{E:Conj}
\aa \opR \bb:= \bb\inv\aa\bb
\end{equation}
for~$\aa, \bb$ in~$\GG$. Then $(\GG, \opR)$ is a quandle, as $\opR$ is idempotent and we find
$$(\aa \opR \bb) \opR \cc = \cc\inv \bb\inv \aa\bb\cc = (\aa \opR \cc) \opR (\bb \opR \cc).$$
The second operation is then given by 
\begin{equation}\label{E:Conjj}
\cc \opRb \bb:= \bb\cc\bb\inv.
\end{equation}
Hereafter, these structures are denoted by~$\Conjj(\GG)$ and~$\Conjjj(\GG)$, according to whether we consider the one operation or the two operations version. 

Variants are obtained by defining $\aa \opR \bb:= \bb^{-\nn}\aa\bb^{\nn}$ for~$\nn$ a fixed integer, or $\aa \opR \bb:= \phi(\bb\inv\aa)\bb$, where $\phi$ is a fixed automorphism of~$\GG$ (one obtains a spindle whenever $\phi$ is an endomorphism).
\end{exam}

\begin{exam}[{\bf core}, or {\bf sandwich}, {\bf quandles}]
Let $\GG$ be a group. For~$\aa, \bb$ in~$\GG$, define $\aa \opR \bb:= \bb\aa\inv\bb$. Then $(\GG, \opR)$ is a quandle, as we find
$$(\aa \opR \bb) \opR \cc = \cc \bb\inv \aa\bb\inv \cc = (\aa \opR \cc) \opR (\bb \opR \cc).$$
The second operation coincides with the first one (``involutory'' quandle). Again variants are possible, involving powers or an involutory automorphism.
\end{exam}

\subsection{More exotic examples}\label{SS:Exotic}

We complete the overview with a few less classical examples. Note that all of them are variations around conjugation.

\begin{exam}[\bf half-conjugacy racks]
Let $\GG$ be a group, and let $\XX$ be a subset of~$\GG$. For~$\aa, \bb$ in~$\GG$ and $\xx, \yy$ in~$\XX$, define 
\begin{equation}\label{E:Half}
(\xx, \aa) \opR (\yy, \bb):= (\xx, \aa\bb\inv \yy \bb).
\end{equation}
Then $(\XX \times \GG, \opR)$ is a rack, as we find
$$((\xx, \aa) \opR (\yy, \bb)) \opR (\zz, \cc) = (\xx, \aa\bb\inv \yy\bb\cc\inv \zz\cc) = ((\xx, \aa) \opR (\zz, \cc)) \opR ((\yy, \bb) \opR (\zz, \cc)).$$
This structure is denoted by~$\HalfConj(\XX,\GG)$. The second operation is then given by 
\begin{equation}\label{E:HalfBis}
(\zz, \cc) \opRb (\yy, \bb):= (\zz, \cc\bb\inv\yy\inv\bb).
\end{equation}
The name is natural, in that, when $\GG$ is a free group based on~$\XX$, the reduced words representing the elements of the associated conjugacy quandle are palindroms~$\aa\inv \cdot \xx \cdot \aa$, and half-conjugacy then corresponds to extracting~$\xx$ and the ``half-word''~$\aa$. The main difference between conjugacy and half-conjugacy is that the latter need not be idempotent: for~$\xx$ in~$\XX$, one finds $(\xx, 1) \opR (\xx, 1) = (\xx, \xx)$ and, more generally,~$(\xx, 1)^{[\nn]} = (\xx, \xx^{\nn - 1})$ for ~$\nn \ge 1$---see~Notation~\ref{N:Power} for the definition of~$\aa^{[\nn]}$. Note that, for for $\GG = (\ZZZZ, +)$ and~$\XX = \{1\}$, one finds $\HalfConj(\XX, \GG) \simeq \Augm$.
\end{exam}

\begin{exam}[\bf injection shelves]
Let~$\XX$ be a non-empty set, and let~$\Sym\XX$ be the group of all bijections from~$\XX$ to itself. Then one can consider the quandle~$\Conj(\Sym\XX)$ as described in Example~\ref{X:Conj}. When the group~$\Sym\XX$ is replaced with the monoid~$\Inj\XX$ of all injections from~$\XX$ to itself, conjugacy makes no more sense, but defining
$$\ff \opR \gg(\xx):= \gg(\ff(\gg\inv(\xx))) \text{ for $\xx \in \Im(\gg)$, and } \ff \opR \gg(\xx):= \xx \text{ otherwise}$$
still provides a shelf, denoted by~$\Conj(\Inj\XX)$. If $\XX$ is finite, $\Inj\XX$ coincides with~$\Sym\XX$, and $\Conj(\Inj\XX)$ is the quandle of Example~\ref{X:Conj}. But, if $\XX$ is infinite, the inclusion of~$\Inj\XX$ in~$\Sym\XX$ is strict, and the shelf~$\Conj(\Inj\XX)$ is neither a spindle nor a rack. Denoting $\XX \setminus \Im(\ff)$ by~$\coIm(\ff)$, one easily checks in~$\Conj(\Inj\XX)$ the following equalities
\begin{equation}\label{E:coIm}
\Im(\ff \opR \gg) = \gg(\Im(\ff)) \cup \coIm(\gg) \text{\quad and \quad }\coIm(\ff \opR \gg) = \gg(\coIm(\ff)).
\end{equation}
As a typical example, consider $\XX:= \ZZZZ_{>0}$, let~$\sh$ be the \emph{shift} mapping~$\nn \mapsto \nn + 1$, and let $\Shift$ be the subshelf of~$\Conj(\Inj\XX)$ generated by~$\sh$, see Fig.~\ref{F:Inj}. Then $\Shift$ is not a spindle, as we have $\sh \opR \sh(1) = 1 \not= \sh(1) = 2$. It is not a rack, either, as we have $\coIm(\sh) = \{1\}$ and, by~\eqref{E:coIm}, $\coIm(\ff \opR \sh) = \sh(\coIm(\ff)) \subseteq \{2, 3 \wdots\}$ for every~$\ff$, so the right translation by~$\sh$ is not surjective.
\end{exam}

\begin{figure}[htb]
\begin{picture}(44,6)(0,0)
\psline(0,4)(43,4)
\psline(0,0)(43,0)
\WPoint(0,0,0)\WPoint(4,0,1)\WPoint(8,0,2)\WPoint(12,0,3)\WPoint(16,0,4)\WPoint(20,0,5)\WPoint(24,0,6)\WPoint(28,0,7)\WPoint(32,0,8)\WPoint(36,0,9)\WPoint(40,0,10)
\WPoint(0,4,0h)\WPoint(4,4,1h)\WPoint(8,4,2h)\WPoint(12,4,3h)\WPoint(16,4,4h)\WPoint(20,4,5h)\WPoint(24,4,6h) \WPoint(28,4,7h) \WPoint(32,4,8h) \WPoint(36,4,9h)\WPoint(40,4,10h)
\ThickSegment(0h,1) \ThickSegment(1h,2) \ThickSegment(2h,3) \ThickSegment(3h,4) \ThickSegment(4h,5) \ThickSegment(5h,6) \ThickSegment(6h,7)\ThickSegment(7h,8)
\ThickSegment(8h,9) \ThickSegment(9h,10) 
\psline(0,4)(43,4)
\psline(0,0)(43,0)
\WPoint(0,0,0)\WPoint(4,0,1)\WPoint(8,0,2)\WPoint(12,0,3)\WPoint(16,0,4)\WPoint(20,0,5)\WPoint(24,0,6)\WPoint(28,0,7)\WPoint(32,0,8)\WPoint(36,0,9)\WPoint(40,0,10)
\WPoint(0,4,0h)\WPoint(4,4,1h)\WPoint(8,4,2h)\WPoint(12,4,3h)\WPoint(16,4,4h)\WPoint(20,4,5h)\WPoint(24,4,6h) \WPoint(28,4,7h) \WPoint(32,4,8h) \WPoint(36,4,9h)\WPoint(40,4,10h)
\put(-2,1){\llap{$\sh:$}}
\end{picture}

\begin{picture}(44,6)(0,0)
\psline(0,4)(43,4)
\psline(0,0)(43,0)
\WPoint(0,0,0)\WPoint(4,0,1)\WPoint(8,0,2)\WPoint(12,0,3)\WPoint(16,0,4)\WPoint(20,0,5)\WPoint(24,0,6)\WPoint(28,0,7)\WPoint(32,0,8)\WPoint(36,0,9)\WPoint(40,0,10)
\WPoint(0,4,0h)\WPoint(4,4,1h)\WPoint(8,4,2h)\WPoint(12,4,3h)\WPoint(16,4,4h)\WPoint(20,4,5h)\WPoint(24,4,6h) \WPoint(28,4,7h) \WPoint(32,4,8h) \WPoint(36,4,9h)\WPoint(40,4,10h)
\ThickSegment(0h,0) \ThickSegment(1h,2) \ThickSegment(2h,3) \ThickSegment(3h,4) \ThickSegment(4h,5) \ThickSegment(5h,6) \ThickSegment(6h,7)\ThickSegment(7h,8)
\ThickSegment(8h,9) \ThickSegment(9h,10) 
\psline(0,4)(43,4)
\psline(0,0)(43,0)
\WPoint(0,0,0)\WPoint(4,0,1)\WPoint(8,0,2)\WPoint(12,0,3)\WPoint(16,0,4)\WPoint(20,0,5)\WPoint(24,0,6)\WPoint(28,0,7)\WPoint(32,0,8)\WPoint(36,0,9)\WPoint(40,0,10)
\WPoint(0,4,0h)\WPoint(4,4,1h)\WPoint(8,4,2h)\WPoint(12,4,3h)\WPoint(16,4,4h)\WPoint(20,4,5h)\WPoint(24,4,6h) \WPoint(28,4,7h) \WPoint(32,4,8h) \WPoint(36,4,9h)\WPoint(40,4,10h)
\put(-2,1){\llap{$\sh \opR \sh:$}}
\end{picture}

\begin{picture}(44,6)(0,0)
\psline(0,4)(43,4)
\psline(0,0)(43,0)
\WPoint(0,0,0)\WPoint(4,0,1)\WPoint(8,0,2)\WPoint(12,0,3)\WPoint(16,0,4)\WPoint(20,0,5)\WPoint(24,0,6)\WPoint(28,0,7)\WPoint(32,0,8)\WPoint(36,0,9)\WPoint(40,0,10)
\WPoint(0,4,0h)\WPoint(4,4,1h)\WPoint(8,4,2h)\WPoint(12,4,3h)\WPoint(16,4,4h)\WPoint(20,4,5h)\WPoint(24,4,6h) \WPoint(28,4,7h) \WPoint(32,4,8h) \WPoint(36,4,9h)\WPoint(40,4,10h)
\ThickSegment(0h,0) \ThickSegment(1h,1) \ThickSegment(2h,3) \ThickSegment(3h,4) \ThickSegment(4h,5) \ThickSegment(5h,6) \ThickSegment(6h,7)\ThickSegment(7h,8)
\ThickSegment(8h,9) \ThickSegment(9h,10) 
\psline(0,4)(43,4)
\psline(0,0)(43,0)
\WPoint(0,0,0)\WPoint(4,0,1)\WPoint(8,0,2)\WPoint(12,0,3)\WPoint(16,0,4)\WPoint(20,0,5)\WPoint(24,0,6)\WPoint(28,0,7)\WPoint(32,0,8)\WPoint(36,0,9)\WPoint(40,0,10)
\WPoint(0,4,0h)\WPoint(4,4,1h)\WPoint(8,4,2h)\WPoint(12,4,3h)\WPoint(16,4,4h)\WPoint(20,4,5h)\WPoint(24,4,6h) \WPoint(28,4,7h) \WPoint(32,4,8h) \WPoint(36,4,9h)\WPoint(40,4,10h)
\put(-2,1){\llap{$(\sh \opR \sh) \opR \sh:$}}
\end{picture}

\begin{picture}(44,6)(0,0)
\psline(0,4)(43,4)
\psline(0,0)(43,0)
\WPoint(0,0,0)\WPoint(4,0,1)\WPoint(8,0,2)\WPoint(12,0,3)\WPoint(16,0,4)\WPoint(20,0,5)\WPoint(24,0,6)\WPoint(28,0,7)\WPoint(32,0,8)\WPoint(36,0,9)\WPoint(40,0,10)
\WPoint(0,4,0h)\WPoint(4,4,1h)\WPoint(8,4,2h)\WPoint(12,4,3h)\WPoint(16,4,4h)\WPoint(20,4,5h)\WPoint(24,4,6h) \WPoint(28,4,7h) \WPoint(32,4,8h) \WPoint(36,4,9h)\WPoint(40,4,10h)
\ThickSegment(0h,2) \ThickSegment(1h,1) \ThickSegment(2h,3) \ThickSegment(3h,4) \ThickSegment(4h,5) \ThickSegment(5h,6) \ThickSegment(6h,7)\ThickSegment(7h,8)
\ThickSegment(8h,9) \ThickSegment(9h,10) 
\psline(0,4)(43,4)
\psline(0,0)(43,0)
\WPoint(0,0,0)\WPoint(4,0,1)\WPoint(8,0,2)\WPoint(12,0,3)\WPoint(16,0,4)\WPoint(20,0,5)\WPoint(24,0,6)\WPoint(28,0,7)\WPoint(32,0,8)\WPoint(36,0,9)\WPoint(40,0,10)
\WPoint(0,4,0h)\WPoint(4,4,1h)\WPoint(8,4,2h)\WPoint(12,4,3h)\WPoint(16,4,4h)\WPoint(20,4,5h)\WPoint(24,4,6h) \WPoint(28,4,7h) \WPoint(32,4,8h) \WPoint(36,4,9h)\WPoint(40,4,10h)
\put(-2,1){\llap{$\sh \opR (\sh \opR \sh):$}}
\end{picture}
\caption{\sf\small A few elements of the shelf~$\Shift$, which is neither a spindle, nor a rack; here, injections go from the top line to the bottom one.}
\label{F:Inj}
\end{figure}
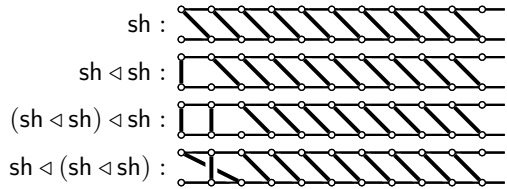

Although its definition is quite simple, very little is known so far about~$\Shift$, see~\cite{Deo}. It is easy to deduce from the criterion of Lemma~\ref{L:Freeness} below that $\Shift$ is not a free shelf. However, almost nothing is known about the following problem:

\begin{ques}
Find a presentation of the shelf~$\Shift$.
\end{ques} 

\begin{exam}[\bf braid shelf]\label{X:Braids}
Artin's braid group~$\Bi$ is the group defined by the (infinite) presentation
\begin{equation}\label{E:BraidPres}
\bigg\langle \sig1, \sig2, ... \ \bigg\vert\ 
\begin{matrix}
\sig\ii \sig j = \sig j \sig\ii 
&\text{for} &\vert i-j \vert\ge 2\\
\sig\ii \sig j \sig\ii = \sig j \sig\ii \sig j 
&\text{for} &\vert i-j \vert = 1
\end{matrix}
\ \bigg\rangle.
\end{equation}
By construction, $\Bi$ is the direct limit when $\nn$ grows to~$\infty$ of the groups~$\BB_\nn$ generated by~$\nn - 1$ generators~$\sig1 \wdots \sig{\nn - 1}$ submitted to the relations of~\eqref{E:BraidPres}, when $\BB_\nn$ is embedded in~$\BB_{\nn + 1}$ by adding~$\sig\nn$. It is well known that the group~$\BB_\nn$ can be realized as the group of isotopy classes of $\nn$~strand braid diagrams, and as the mapping class group of an $\nn$-punctured disk---see for instance~\cite{Bir} or~\cite{Dhr}. It directly follows from the presentation~\eqref{E:BraidPres} that the map~$\sh: \sig\ii \mapsto\nobreak \sig{\ii + 1}$ extends into a non-surjective endomorphism of~$\Bi$ (``\emph{shift} endomorphism'').

Define on~$\Bi$ a binary operation~$\opR$ by
\begin{equation}\label{E:BraidShelf}
\aa \opR \bb:= \sh(\bb)\inv \, \sig1 \, \sh(\aa) \, \bb,
\end{equation}
a shifted conjugation with the factor~$\sig1$ added. Then one checks the equalities
\begin{equation}
(\aa \opR \bb) \opR \cc = \sh(\cc)\inv \, \sh^2(\bb)\inv \, \sig1\sig2 \, \sh^2(\aa) \, \sh(\bb) \, \cc = (\aa \opR \cc) \opR (\bb \opR \cc),
\end{equation}
so $(\Bi, \opR)$ is a shelf. It is easy to check that this shelf is neither a spindle, nor a rack: for instance, we have $1 \opR 1 = \sig1 \not= 1$, and $\aa \opR 1 = 1$ is impossible, as we have $\aa \opR 1 = \sig1 \, \sh(\aa)$, and $\sig1$ does not belong to the image of~$\sh$. For more information about the braid shelf, we refer to Section~\ref{SS:Semantic} below and to~\cite{Dja}. 
\end{exam}

We can try to further generalize Example~\ref{X:Braids}. Let $\GG$ be a group, let~$\ss$ a fixed element of~$\GG$, and let $\phi$ is an endomorphism of~$\GG$. Copying~\eqref{E:BraidShelf}, define
\begin{equation}
\aa \opR \bb:= \phi(\bb)\inv \, \ss \, \phi(\aa) \, \bb.
\end{equation}
Consider~$(\GG, \opR)$. This is a shelf if, and only if, the element~$\ss$ commutes with every element in the image of~$\phi^2$ and satisfies the relation $\ss\phi(\ss)\ss = \phi(\ss)\ss\phi(\ss)$, which means that the subgroup of~$\GG$ generated by~$\ss$ is a homomorphic image of the braid group~$\Bi$. Thus, essentially, the above construction works only for the braid group and its quotients.

A typical example appears when \eqref{E:BraidShelf} is applied in the quotient~$\Sym\infty$ of~$\Bi$ obtained by adding the involutivity relations~$\sigg12 = 1$: then $\Sym\infty$ is the group of all permutations of~$\ZZZZ_{>0}$ that move finitely many points only. {\it Mutatis mutandis} (left-shelves are considered), it is proved in~\cite[Cor.~I.4.18]{Dgd}  that mapping~$\ff$ to the composed map~$\sh \comp \ff$ defines an (injective) homomorphism from~$(\Sym\infty, \opR)$ into the shelf~$\Conj(\Inj{\ZZZZ_{>0}})$.

Extensions and variations of the braid shelf appear in~\cite{Dfn, Dgb, Dhe}.

\begin{exam}[\bf iterations of an elementary embedding]\label{X:Emb}
Large cardinal axioms play a central r\^ole in modern set theory. A number of such axioms involve elementary embeddings, which are mappings from a set to itself that preserve all notions that are definable in first order logic from the membership relation~$\in$. For every ordinal~$\alpha$, one denotes by~$\VV_\alpha$ the set obtained from~$\varnothing$ by applying $\alpha$~times the powerset operation. Let~$\mathcal{E}_\lambda$ denote the family of all elementary embeddings from~$\VV_\lambda$ to itself. One says that an ordinal~$\lambda$ is a \emph{Laver cardinal} if $\mathcal{E}_\lambda$ contains at least one element that is not the identity. The point is as follows. The specific properties of the sets~$\VV_\alpha$ imply that, if $\ii$ and~$\jj$ belong to~$\mathcal{E}_\lambda$, then one can \emph{apply}~$\ii$ to~$\jj$ and obtain a new element~$\ii[\jj]$ of~$\mathcal{E}_\lambda$. Moreover, for~$\ii, \jj, \kk, \ell$ in~$\mathcal{E}_\lambda$, if $\ell = \jj[\kk]$ holds, then so does $\ii[\ell] = \ii[\jj][\ii[\kk]]$, that is
\begin{equation}\label{E:LDSet}
\ii[\jj[\kk]] = \ii[\jj][\ii[\kk]].
\end{equation}
The reason for that is that $\ii$ preserves every definable notion, in particular the operation of applying a function to an argument. Then \eqref{E:LDSet} means that the binary operation $(\ii, \jj) \mapsto \ii[\jj]$ on~$\mathcal{E}_\lambda$ obeys the selfdistributivity law~$\LD$. When $\jj$ is the identity mapping, the closure of~$\{\jj\}$ under the ``application'' operation is trivial. But assume that $\lambda$ is a Laver cardinal, and $\jj$ is an element of~$\mathcal{E}_\lambda$ that is not the identity. Then the closure of~$\{\jj\}$ under the ``application'' operation is a nontrivial left-shelf~$\Iter(\jj)$. This  left-shelf has fascinating properties, yet its structure is still far from being understood completely~\cite{Lvb, Dou, DoJ, Jech, Dgs, Dimo}. 
\end{exam}

\begin{exam}[\bf Laver tables]
For every positive integer~$\NN$, there exists a unique binary operation~$\op$ on~$\{1 \wdots \NN\}$ that obeys the law
$$\xx \op (\yy \op 1) = (\xx \op \yy) \op (\xx \op 1),$$
and the structure so obtained is a left-shelf if, and only if, $\NN$ is a power of~$2$. The structure with~$2^\nn$~elements is called the \emph{$\nn$th Laver table}, usually denoted by~$\AA_\nn$. Laver tables appear as the elementary building bricks for constructing all (finite) monogenerated shelves~\cite{Drd, Drf, Smed}, and can adequately be seen as counterparts of cyclic groups in the SD-world. We refer to the survey by A.\,Dr\'apal in this volume for a more complete introduction~\cite{DraS}. Let us simply mention here that Laver tables are natural quotients of the left-shelf~$\Iter(\jj)$ of Example~\ref{X:Emb}~\cite{Lvd}, and that some of their combinatorial properties are so far established only assuming the existence of a Laver cardinal, which is an unprovable axiom, see for instance~\cite{Dgs}.
\end{exam}

\subsection{An overview of the SD-world}\label{SS:Chart}

We present in Fig.~\ref{F:Chart} a sort of chart of the SD-world as we can conceive it now. Here we leave the case of spindles aside, and, therefore, we have three classes included one in the other, namely quandles, racks, and shelves. The lanscape is rather different according to whether we consider quandles and racks, or general shelves, and, on the other hand, whether we consider monogenerated structures, or structures with more than one generator.

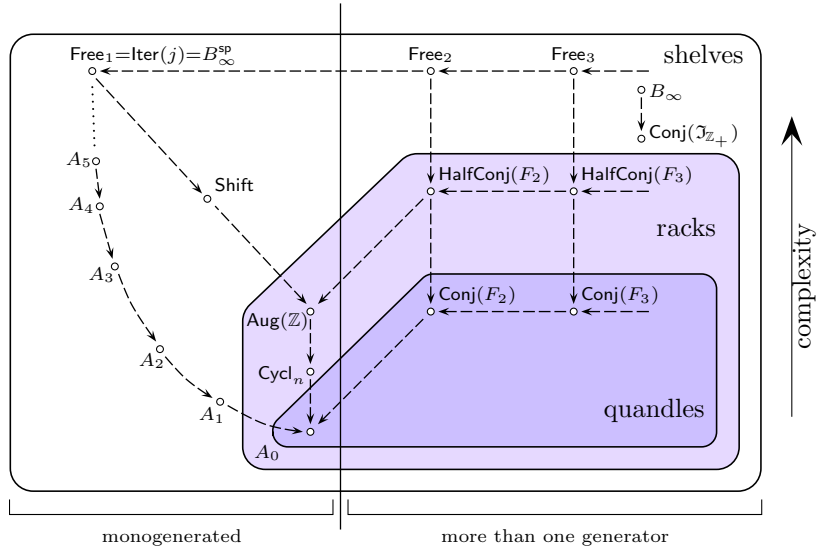
\begin{figure}[htb]
\begin{picture}(100,70)(2,-5)
\everymath{\color{black}}%
\psset{unit=1mm}
\psframe[linewidth=0.6pt,framearc=.1](0,0)(100,61)
\pspolygon[linearc=3,linewidth=0.6pt,fillstyle=solid,fillcolor=PcolorAAA](31,3)(97,3)(97,45)(53,45)(31,24)
\pspolygon[linearc=2,linewidth=0.6pt,fillstyle=solid,fillcolor=PcolorAA](35,6)(94,6)(94,29)(55.5,29)(35,9)
\put(87,57){{shelves}}
\put(86,34){{racks}}
\put(79,10){{quandles}}

\psline[linewidth=.5pt](44,65)(44,-5)
\psline[linewidth=.4pt](0,-1)(0,-3)(43,-3)(43,-1)
\pcline[linewidth=.4pt](0,-3)(43,-3)\tbput{\scriptsize monogenerated}
\psline[linewidth=.4pt](45,-1)(45,-3)(100,-3)(100,-1)
\pcline[linewidth=.4pt](45,-3)(100,-3)\tbput{\scriptsize more than one generator}

\pcline[linewidth=.4pt,arrowsize=8pt]{->}(104,10)(104,50)
\put(104,19){\rotatebox{90}{\hbox{complexity}}}

\pscircle[linewidth=.4pt,fillstyle=solid](40,8){1.5pt}
\put(32.5,4.5){$\scriptstyle\AA_0$}

\pscircle[linewidth=.4pt,fillstyle=solid](56,24){1.5pt}
\put(57,25.5){$\scriptstyle\Conj(\FF_2)$}
\pscircle[linewidth=.4pt,fillstyle=solid](75,24){1.5pt}
\put(76,25.5){$\scriptstyle\Conj(\FF_3)$}

\psline[style=thinexist]{->}(55,23)(41,9)
\psline[style=thinexist]{->}(85,24)(76,24)
\psline[style=thinexist]{->}(74,24)(57,24)

\pscircle[linewidth=.4pt,fillstyle=solid](40,24){1.5pt}
\put(31.5,22){$\scriptstyle\Augm$}
\pscircle[linewidth=.4pt,fillstyle=solid](40,16){1.5pt}
\put(33,15){$\scriptstyle\Cycl_\nn$}
\psline[style=thinexist]{->}(40,23)(40,17)
\psline[style=thinexist]{->}(40,15)(40,9)

\pscircle[linewidth=.4pt,fillstyle=solid](56,40){1.5pt}
\put(57,41.5){$\scriptstyle\HalfConj(\FF_2)$}
\psline[style=thinexist]{->}(55,39)(41,25)
\psline[style=thinexist]{->}(56,39)(56,25)

\pscircle[linewidth=.4pt,fillstyle=solid](75,40){1.5pt}
\put(76,41.5){$\scriptstyle\HalfConj(\FF_3)$}
\psline[style=thinexist]{->}(85,40)(76,40)
\psline[style=thinexist]{->}(74,40)(57,40)
\psline[style=thinexist]{->}(75,39)(75,25)

\pscircle[linewidth=.4pt,fillstyle=solid](11,56){1.5pt}
\put(8,57.5){$\scriptstyle\Free_1 = \Iter(\jj) = \BRsp$}
\psline[style=thinexist]{->}(11.5,55)(39.5,25)

\pscircle[linewidth=.4pt,fillstyle=solid](56,56){1.5pt}
\put(53,57.5){$\scriptstyle\Free_2$}
\psline[style=thinexist]{->}(55,56)(12,56)
\psline[style=thinexist]{->}(56,55)(56,41)

\pscircle[linewidth=.4pt,fillstyle=solid](75,56){1.5pt}
\put(72,57.5){$\scriptstyle\Free_3$}
\psline[style=thinexist]{->}(75,55)(75,41)
\psline[style=thinexist]{->}(74,56)(57,56)
\psline[style=thinexist]{->}(85,56)(76,56)

\pscircle[linewidth=.4pt,fillstyle=solid](28,12){1.5pt}
\put(25.2,9.5){$\scriptstyle{\AA_1}$}
\pscurve[style=thinexist]{->}(29,11.5)(34,9)(39,8)

\pscircle[linewidth=.4pt,fillstyle=solid](20,19){1.5pt}
\put(17.2,16.8){$\scriptstyle{\AA_2}$}
\pscurve[style=thinexist]{->}(20.7,18)(23.5,15)(27,12.5)

\pscircle[linewidth=.4pt,fillstyle=solid](14,30){1.5pt}
\put(10.5,28.2){$\scriptstyle{\AA_3}$}
\pscurve[style=thinexist]{->}(14.5,29)(16.2,25)(19.5,20)

\pscircle[linewidth=.4pt,fillstyle=solid](12,38){1.5pt}
\put(7.8,37.5){$\scriptstyle{\AA_4}$}
\psline[style=thinexist]{->}(12,37)(13.7,31)

\pscircle[linewidth=.4pt,fillstyle=solid](11.5,44){1.5pt}
\put(7.5,43.5){$\scriptstyle{\AA_5}$}
\psline[style=thinexist]{->}(11.5,43)(12,39)
\psline[linestyle=dotted,dotsep=2pt](11,54)(11.3,46)

\pscircle[linewidth=.4pt,fillstyle=solid](84,53.5){1.5pt}
\put(85,52.5){$\scriptstyle\Bi$}
\pscircle[linewidth=.4pt,fillstyle=solid](84,47){1.5pt}
\put(85,47){$\scriptstyle\Conj(\Inj{\ZZZZ_{+}})$}
\psline[style=thinexist]{->}(84,52.5)(84,48)

\psline[linewidth=2pt,linecolor=white](11.5,55)(27.3,38)
\pscircle[linewidth=.4pt,fillstyle=solid](26.3,39){1.5pt}
\put(27.3,40){$\scriptstyle\Shift$}
\psline[style=thinexist]{->}(11.5,55)(25.6,39.7)
\end{picture}
\caption[]{\sf\small A chart of the SD-world, with dashed arrows for quotients. As every shelf with~$\nn$ generators is a quotient of the free shelf~$\Free_\nn$ on~$\nn$~generators, it is natural to put~$\Free_\nn$ on the top of the diagram; as will be seen in Section~\ref{SS:TwoGener}, $\Free_\nn$ is not really more complicated than~$\Free_1$ for~$\nn \ge 2$, so we put them on the same complexity line. We represent similarly free racks, that is, half-conjugacy racks associated with free groups, and free quandles, that is, conjugacy quandles associated with free groups (Prop.~\ref{P:FreeConj}). Here, the complexity drops for one generator. On the  monogenerated side, quandles and racks are almost trivial, whereas shelves are many, with a spine made by Laver tables, which form an inverse system with limit~$\Free_1$ (whenever a Laver cardinal exists); by contrast, on the multigenerated case, lots of racks and quandles have been investigated, whereas not so many really new examples of shelves are known.}
\label{F:Chart}
\end{figure}

\section{Word problem, the case of shelves I}\label{S:WPLD1}

For every algebraic law or family of algebraic laws, a basic question is the associated word problem, namely the question of deciding whether or not two terms in the absolutely free algebra for the relevant operations are equivalent with respect to the congruence generated by the law(s). In other words, whether or not they represent the same element in the corresponding free structure. Thus four word problems respectively involving shelves, spindles, racks, and quandles occur. In this section and the next one, we begin with the case of shelves, that is, of selfdistributivity alone: in this case, several solutions are known, none of which is trivial, and we explain them. Different solutions can be considered: \emph{syntactic} solutions aim at manipulating terms and deciding their possible equivalence directly, whereas \emph{semantic} solutions consist in evaluating terms in some particular concrete structure(s): when the latter is free, or includes a free structure, two terms are equivalent if, and only if, their evaluations coincide. For instance, in the case of associativity, a syntactic solution may consist in removing parentheses and checking whether the remaining words coincide, whereas a semantic solution may consist in evaluating terms in the free monoids~$\XX^*$---essentially the same thing in this trivial case. 

This first part about word problems is divided into three sections. Section~\ref{SS:Comparison} is preparatory and explains ``comparison property'', an important feature of selfdistributivity. In Section~\ref{SS:CondSyntactic}, we derive a conditional solution for the word problem in the case of one variable, namely one that is valid provided there exists a shelf with a certain ``acyclicity'' property. Then, in Section~\ref{SS:Syntactic}, we describe two examples of such shelves and thus solve the word problem. Finally, we describe in Section~\ref{SS:Semantic} a semantic solution based on the braid shelf of Example~\ref{X:Braids} that is more efficient than the syntactic solution so far considered.

\subsection{The comparison property}\label{SS:Comparison}

To make our survey compatible with the existing literature~\cite{Dgd, Lvb}, we switch hereafter to the left version of selfdistributivity~$\LD$ and, accordingly, use~$\op$ instead of~$\opR$.

We denote by~$\TERM\op\XX$ the family of all well formed terms constructed from a set~$\XX$ (usually $\{\xx_1, \xx_2, ...\}$, with elements called \emph{variables}) using the operation~$\op$, \ie, the absolutely free $\op$-algebra based on~$\XX$. We denote by~$\eqLD$ the congruence on $\TERM\op\XX$ generated by the instances of the law~$\LD$, \ie, the least equivalence relation that is compatible with the operations and contains the said instances. By construction, the quotient-structure~$\TERM\op\XX{/}{\eqLD}$ is a left-shelf generated by~$\XX$, and it is universal for all such left-shelves, so it is a free left-shelf based on~$\XX$.

We begin with a preparatory result about selfdistributivity in the case of terms in one variable. To state it, we start with the natural notion of a \emph{divisor}.

\begin{defi}[\bf division relation]
For $\op$ a binary operation on~$\SS$ and~$\aa, \bb$ in~$\SS$, we say that $\aa$ \emph{divides}~$\bb$, written $\aa \div \bb$, if $\aa \op \xx = \bb$ holds for some~$\xx$. We write~$\divs$ for the transitive closure of~$\div$.
\end{defi}

If $\op$ is associative, there is no need to distinguish between~$\div$ and~$\divs$, since we then have $(\aa \op \xx_1) \op \xx_2 = \aa \op (\xx_1 \op \xx_2)$, but, in general, $\div$ need not be transitive.

The following result about selfdistributivity is then fundamental:

\begin{lemm}[\bf comparison property]\label{L:Comp}
If $\SS$ is a monogenerated left-shelf and~$\aa, \bb$ belong to~$\SS$, then at least one of $\aa \divs \bb$, $\aa = \bb$, $\bb \divs \aa$ holds.
\end{lemm}

If $\phi$ is a morphism, $\aa \divs \bb$ implies $\phi(\aa) \divs \phi(\bb)$, so the point is to prove Lemma~\ref{L:Comp} when $\SS$ is a free left-shelf generated by a single element, say~$\xx$. By definition, the latter consists of the $\eqLD$-classes of terms in~$\TERM\op\xx$---we write $\TERM\op\xx$ for $\TERM{\op}{\{\xx\}}$.

\begin{nota}[\bf relation~$\divLD$]
For $\TT, \TT'$ in~$\TERM\op\xx$, we write $\TT \divLD \TT'$ if there exists~$\TT_1$ satisfying $\TT \op \TT_1 \eqLD \TT'$, and $\divsLD$ for the transitive closure of~$\divLD$. 
\end{nota}

Then, an equivalent form of Lemma~\ref{L:Comp} is

\begin{lemm}[{\bf comparison property}, reformulated]\label{L:Comp2}
For all~$\TT, \TT'$ in~$\TERM\op\xx$, at least one of $\TT \divsLD \TT'$, $\TT \eqLD \TT'$, or $\TT' \divsLD \TT$ holds. 
\end{lemm}

Elements of~$\TERM\op\XX$ can be seen as binary trees with internal nodes labeled~$\op$ and leaves labeled with elements of~$\XX$ (thus variables). What Lemma~\ref{L:Comp2} says is that, if $\TT$ and~$\TT'$ are any two terms in one variable, then, up to $\LD$-equivalence, one is always an iterated left subterm of the other. When associativity is considered, the result is trivial, since a term in one variable is just a power. In the case of selfdistributivity, the result, which is difficult, was proved in~\cite{Deq}, and independently reproved shortly after by R.\,Laver in~\cite{Lvb} using a disjoint argument. Here we sketch the two steps of the former proof, which is simpler.

First, by definition, the relation~$\eqLD$ is the congruence on~$\TERM\op\XX$ generated by all pairs of terms of the form 
\begin{equation}\label{E:LDTerms}
(\ \TT_1 \op (\TT_2 \op \TT_3)\ , \ (\TT_1 \op \TT_2) \op (\TT_1 \op \TT_3)\ ).
\end{equation}
We consider an oriented, non-symmetric version of the latter relation.

\begin{defi}[\bf LD-expansion]
We let~$\expLD$ be the smallest reflexive and transitive relation on~$\TERM\op\XX$ that is compatible with multiplication and contains all pairs~\eqref{E:LDTerms}. When~$\TT \expLD \TT'$ holds, we say that $\TT'$ is an \emph{LD-expansion} of~$\TT$.
\end{defi}

Thus, $\TT'$ is an LD-expansion of~$\TT$ if $\TT'$ can be obtained from~$\TT$ by applying the LD law, but always in the expanding direction. Clearly, $\TT \expLD \TT'$ implies~$\TT \eqLD \TT'$, but the converse implication fails, as $\expLD$ is not symmetric: $\xx \op (\xx \op \xx) \expLD (\xx \op \xx) \op (\xx \op \xx)$ holds, but $(\xx \op \xx) \op (\xx \op \xx) \expLD \xx \op (\xx \op \xx)$ does not.

\begin{lemm}[\bf confluence property]\label{L:Confl}
Two LD-equivalent terms admit a common LD-expansion.
\end{lemm}

\begin{proof}[Idea of the proof]
By definition, two terms~$\TT, \TT'$ are LD-equivalent if, and only if, there exists a finite zigzag of LD-expansions and inverses of LD-expansions connecting~$\TT$ to~$\TT'$. The point is to show that there always exists a zigzag with only one expansion and one inverse of expansion. To this end, it suffices to prove that the relation~$\expLD$ is confluent, \ie, that any two LD-expansions~$\TT', \TT''$ of a term~$\TT$ admit a common LD-expansion. It is easy to check local confluence, namely confluence when $\TT'$ and $\TT''$ are obtained from~$\TT$ by applying the LD law at most once (``atomic'' LD-expansion). But a termination problem arises, because the relation~$\expLD$ is far from being noetherian (infinite sequences of LD-expansions exist) and Newman's standard diamond lemma~\cite{New} cannot be applied. To solve this, one considers, for every term~$\TT$, the term~$\partial\TT$ inductively defined by
\begin{equation}\label{E:Der}
\partial\TT:= 
\begin{cases}
\xx&\text{for $\TT = \xx$},\cr
\partial\TT_0 \otimes \partial\TT_1&\text{for $\TT = \TT_0 \op \TT_1$},
\end{cases}
\end{equation}
where $\otimes$ itself is inductively defined by
\begin{equation}\label{E:Otimes}
\SS \otimes \TT := 
\begin{cases}
\SS \op \TT&\text{for $\TT = \xx$},\cr
(\SS \otimes \TT_0) \op (\SS \otimes \TT_1)&\text{for $\TT = \TT_0 \op \TT_1$}.
\end{cases}
\end{equation}
One can show that $\partial\TT$ is an LD-expansion of~$\TT$ and of all atomic LD-expansions of~$\TT$, and that $\TT \expLD \TT'$ implies $\partial\TT \expLD \partial\TT'$. It is then easy to deduce that, for every~$\pp$, the term~$\partial^\pp\TT$ is an LD-expansion of all LD-expansions of~$\TT$ obtained using at most~$\pp$~atomic expansion steps. From there, any two LD-expansions of~$\TT$ admit as a common LD-expansion any term~$\partial^\pp\TT$ with~$\pp$ large enough.
\end{proof}

The second ingredient is the following specific property.

\begin{nota}[\bf powers]\label{N:Power}
For $\op$ a binary operation on~$\SS$ and $\aa$ in~$\SS$, the \emph{left} and \emph{right} powers of~$\aa$ are defined by $\aa_{[1]} := \aa^{[1]} := \aa$, and $\aa_{[\nn + 1]}:= \aa_{[\nn]} \op \aa$  and $\aa^{[\nn + 1]} := \aa \op \aa^{[\nn]}$.
\end{nota}

\begin{lemm}[\bf absorption property]\label{L:Absorp}
If $\SS$ is a left-shelf generated by an element~$\gg$, then, for every~$\aa$ in~$\SS$, we have $\aa \op \gg^{[\nn]} = \gg^{[\nn + 1]}$ for~$\nn$ large enough (depending on~$\aa$).
\end{lemm}

\begin{proof}
As in the case of the comparison property, it is sufficient to consider the case of the free left-shelf, namely to establish that, for every term~$\TT$ in~$\TERM\op\xx$, 
\begin{equation}\label{E:Absorp}
\TT \op \xx^{[\nn]} \eqLD \xx^{[\nn + 1]}
\end{equation}
holds for~$\nn$ large enough. We use induction on (the size of)~$\TT$. For~$\TT = \xx$, \eqref{E:Absorp} holds for every~$\nn \ge 1$, with $\eqLD$ being an equality. Otherwise, write $\TT = \TT_0 \op \TT_1$, and assume that (the counterpart of)~\eqref{E:Absorp} holds for~$\TT_\ii$ for every~$\nn \ge \nn_\ii$. For $\nn \ge \max(\nn_0, \nn_1) + 1$, we obtain
\begin{align*}
\xx^{[\nn + 1]} 
&\eqLD \TT_0 \op \xx^{[\nn]} 
&&\text{owing to $\nn \ge \nn_0$ and the IH for~$\TT_0$,}\\
&\eqLD \TT_0 \op (\TT_1 \op \xx^{[\nn-1]}) 
&&\text{owing to $\nn-1 \ge \nn_1$ and the IH for~$\TT_1$,}\\
&\eqLD (\TT_0 \op \TT_1) \op (\TT_0 \op \xx^{[\nn-1]}) 
&&\text{by LD,}\\
&\eqLD (\TT_0 \op \TT_1) \op \xx^{[\nn]} = \TT \op \xx^{[\nn]}
&&\text{owing to $\nn-1 \ge \nn_0$ and the IH for~$\TT_0$. \hfill $\square$}
\end{align*}
\let\qed=\relax
\end{proof}

Using Lemmas~\ref{L:Confl} and~\ref{L:Absorp}, it is now easy to deduce the comparison property.

\begin{proof}[Proof of Lemma~\ref{L:Comp2}]
(See Fig.~\ref{F:Comp}.) Let $\TT, \TT'$ belong to~$\TERM\op\xx$. By Lem\-ma~\ref{L:Absorp}, for $\nn$ large enough, we have
\begin{equation}
\TT \op \xx^{[\nn]} \eqLD \xx^{[\nn + 1]} \eqLD \TT' \op \xx^{[\nn]}.
\end{equation}
By Lemma~\ref{L:Confl}, we deduce that $\TT \op \xx^{[\nn]}$ and $\TT' \op \xx^{[\nn]}$ admit a common LD-expansion~$\TT''$ (which can be assumed to be~$\partial^\pp\xx^{[\nn + 1]}$ for some~$\pp$). Using an induction on the number of LD-expansion steps, it is easy to verify that, if $\TT$ is not a variable and $\TT'$ is an LD-expansion of~$\TT$, then there exists~$\rr$ such that the $\rr$th iterated left subterm~$\LS^{\rr}(\TT')$ of~$\TT'$ is an LD-expansion of the left subterm~$\LS(\TT)$ of~$\TT$. In the current case, the left subterm of~$\TT \op \xx^{[\nn]}$ is~$\TT$, and we deduce that there exist~$\rr$ satisfying $\TT \expLD \LS^\rr(\TT'')$, whence $\TT \eqLD \LS^\rr(\TT'')$. Similarly, there exist~$\rr'$ satisfying $\TT' \expLD \LS^{\rr'}(\TT'')$, whence $\TT' \eqLD \LS^{\rr'}(\TT'')$. 

Now three cases may occur. For $\rr = \rr'$, we find $\TT \eqLD \LS^{\rr}(\TT) \eqLD \TT'$, whence $\TT \eqLD \TT'$. For $\rr > \rr'$, the term~$\LS^{\rr}(\TT)$ is an iterated left subterm of~$\LS^{\rr'}(\TT)$, that is, we have $\LS^{\rr}(\TT) \divs \LS^{\rr'}(\TT)$, hence $\TT \divsLD \TT'$. Similarly, for~$\rr <\nobreak \rr'$, we obtain~$\TT \multsLD \TT'$ by a symmetric argument.
\end{proof}

\begin{figure}[htb]
\begin{picture}(100,63)(0,-3)
\put(15,42){\begin{picture}(0,0)(0,0)
\setlength\unitlength{0.8mm}\psset{unit=.8mm}
\psline[fillstyle=solid,fillcolor=PcolorAAA](10,18)(0,0)(20,0)(10,18)
\put(8,6){$\TT$}
\psline(10,18)(13,22)(16,18)
\psline(16,18)(14,15)(16,18)(18,15)(16,18)(18,15)(16,12)(18,15)(20,12)(18,9)(20,12)(22,9)(20,6)(22,9)(24,6)(22,3)(24,6)(26,3)(24,0)(26,3)(28,0)
\end{picture}}

\put(75,42){\begin{picture}(0,0)(0,0)
\setlength\unitlength{0.8mm}\psset{unit=.8mm}
\psline[fillstyle=solid,fillcolor=PcolorBBB](10,18)(0,0)(20,0)(10,18)
\put(8,6){$\TT'$}
\psline(10,18)(13,22)(16,18)
\psline(16,18)(14,15)(16,18)(18,15)(16,18)(18,15)(16,12)(18,15)(20,12)(18,9)(20,12)(22,9)(20,6)(22,9)(24,6)(22,3)(24,6)(26,3)(24,0)(26,3)(28,0)
\end{picture}}

\put(40,42){\begin{picture}(0,0)(0,0)
\setlength\unitlength{0.8mm}\psset{unit=.8mm}
\psline(12,18)(14,21)(16,18)
\psline(16,18)(14,15)(16,18)(18,15)(16,18)(18,15)(16,12)(18,15)(20,12)(18,9)(20,12)(22,9)(20,6)(22,9)(24,6)(22,3)(24,6)(26,3)(24,0)(26,3)(28,0)
\put(0,12){$\eqLD$}
\put(32,12){$\eqLD$}
\end{picture}}

\put(35,3){\begin{picture}(0,0)(0,3)
\psset{unit=1mm}
\psline(20,36)(0,0)(40,0)(20,36)
\put(17.5,29){$\TT''$}
\psline[arrowsize=5pt]{->}(5,40)(15,32)
\psline[arrowsize=5pt]{->}(35,40)(25,32)

\psline[fillstyle=solid,fillcolor=PcolorAAA](15,27)(0,0)(30,0)(15,27)
\WDot{15,27}\put(11,28){{$0^{\rr}$}}
\psbezier[arrowsize=5pt,linestyle=dashed]{->}(-12,41)(-12,25)(0,32)(10,24)

\psline[fillstyle=solid,fillcolor=PcolorBBB](10,18)(0,0)(20,0)(10,18)
\put(3,2){$\LS^{\rr'}(\TT'')$}
\WDot{10,18}\put(6,19){{$0^{\rr'}$}}
\psbezier[arrowsize=5pt,linestyle=dashed,border=2pt](48,41)(48,25)(35,37)(20,37)
\psbezier[arrowsize=5pt,linestyle=dashed,border=2pt]{->}(20,37)(5,37)(-5,25)(5,15)
\psline[style=thin](0,-1)(0,-2)(30,-2)(30,-1)
\put(9,-5.5){$\LS^{\rr}(\TT'')$}
\end{picture}}
\end{picture}
\caption{\sf\small Proof of the comparison property: the terms~$\TT$ and~$\TT'$ admit LD-expansions that both are iterated left subterms of some (large) term~$\TT''$: the latter coincide, or one is an iterated left subterm of the other.}
\label{F:Comp}
\end{figure}
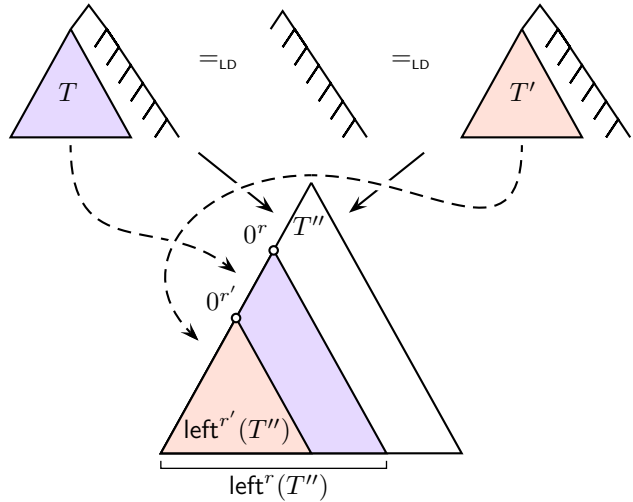

\subsection{A conditional syntactic solution}\label{SS:CondSyntactic}

We are now ready to describe a syntactic solution of the word problem for~$\LD$. However, in a first step, the solution will remain conditional, as its correctness relies on an extra assumption that will be established in the next section only.

\begin{defi}[\bf acyclic]
A left-shelf~$\SS$ is called \emph{acyclic} if the relation~$\div$ on~$\SS$ has no cycle.
\end{defi}

\begin{lemm}[\bf exclusion property]\label{L:Excl1}
If there exists an acyclic left-shelf, the relations $\eqLD$ and~$\divsLD$ on~$\TERM\op\xx$ exclude one another.
\end{lemm}

\begin{proof}
Assume that $\SS$ is an acyclic left-shelf, and~$\TT, \TT'$ are terms in~$\TERM\op\xx$ satisfying $\TT \divsLD \TT'$. Let~$\gg$ be an element of~$\SS$, and let~$\TT(\gg)$ and~$\TT'(\gg)$ be the evaluations of~$\TT$ and~$\TT'$ in~$\SS$ when $\xx$ is given the value~$\gg$. By definition, $\TT \divsLD \TT'$ implies $\TT(\gg) \divs \TT'(\gg)$ in~$\SS$, whence $\TT(\gg) \not= \TT'(\gg)$, since $\div$ has no cycle in~$\SS$. As $\TT \eqLD \TT'$ would imply $\TT(\gg) = \TT'(\gg)$, we deduce that $\TT \eqLD \TT'$ is impossible.
\end{proof}

\begin{prop}[\bf conditional word problem]\cite{Deq, Lvb}\label{P:CondSol}
If there exists an acyclic left-shelf, the word problem of~$\LD$ is decidable in the case of one variable.
\end{prop}

\begin{proof}
The relation~$\eqLD$ on~$\TERM\op\xx$ is semi-decidable, meaning that there exists an algorithm that, starting with any two terms~$\TT, \TT'$, returns~$\True$ if $\TT \eqLD\nobreak \TT'$ holds, and runs forever if $\TT \eqLD \TT'$ fails: just start with~$\TT$, ``stupidly'' enumerate all terms~$\TT''$ that are $\LD$-equivalent to~$\TT'$ by repeatedly applying~$\LD$ in either direction at any position, and test $\TT =? \TT''$. Similarly, $\divsLD$ is semi-decidable: starting with~$\TT, \TT'$, enumerate all terms~$\TT''$ that are $\LD$-equivalent to~$\TT'$ and test $\TT \divs? \TT''$, \ie, test whether $\TT$ is a proper iterated left subterm of~$\TT''$. Thus, both $\eqLD$ and $\divsLD \cup \multsLD$ are semi-decidable relations on~$\TERM\op\xx$. By the comparison property (Lemma~\ref{L:Comp2}), their union is all of~$\TERM\op\xx$ and, if there exists an acyclic left-shelf, they are disjoint by the exclusion property (Lemma~\ref{L:Excl1}). Thus, in this case, $\divsLD {\cup} \multsLD$ coincides with~$\not\eqLD$. Then $\eqLD$ and its complement are semi-decidable, hence they are decidable.
\end{proof}

The approach leads to a (certainly very inefficient) algorithm. Fix an exhaustive enumeration~$(\ss_\nn, \ss'_\nn)_{\nn \ge 0}$ of the pairs of finite sequences of positions in a binary tree---that is, sequences of binary addresses, see Section~\ref{SS:Thompson} below---and let $\EXPAND(\TT, \ss)$ be the result of applying~LD in the expanding direction starting from~$\TT$ and according to the sequence of positions~$\ss$.

\begin{algo}[\bf syntactic solution, case of one variable]\label{A:Syntact}

\rm\hfill\par

\noindent {\bf Input}: Two terms $\TT$, $\TT'$ in~$\TERM\op\xx$

\noindent {\bf Output}: $\True$ if $\TT$ and $\TT'$ are $\LD$-equivalent, $\False$ otherwise

\lab{1} $\Found := \False$

\lab{2} $\nn := 0$

\lab{3} {\bf while not} $\Found$ {\bf do}

\lab{4} \quad $\TT_1:= \EXPAND(\TT, \ss_\nn)$

\lab{5} \quad $\TT'_1:= \EXPAND(\TT', \ss'_\nn)$

\lab{6} \quad {\bf if} $\TT_1 = \TT'_1$ {\bf then}

\lab{7} \quad\quad {\bf return} $\True$

\lab{8} \quad\quad $\Found := \True$

\lab{9} \quad {\bf if} $\TT_1 \divs \TT'_1$ {\bf or} $\TT \mults \TT'_1$ {\bf then}

\lab{10} \quad\quad {\bf return} $\False$

\lab{11} \quad\quad $\Found := \True$

\lab{12} \quad $\nn:= \nn + 1$
\end{algo}

Termination of the algorithm follows from the comparison property: for all~$\TT, \TT'$, there must exist a pair of sequences~$(\ss, \ss')$ such that expanding~$\TT$ according to~$\ss$ and~$\TT'$ according to~$\ss'$ yields a pair of terms~$(\TT_1, \TT'_1)$ consisting of equal or comparable terms; correctness follows from the exclusion property, which guarantees that $\TT_1 \divs \TT'_1$ implies $\TT_1 \not\eqLD \TT'_1$, hence $\TT \not\eqLD \TT'$. So, at this point, we may state:

\begin{prop}[\bf conditional syntactic solution, case of one variable]
If there exists an acyclic left-shelf, Algorithm~\ref{A:Syntact} is correct, and the word problem of selfdistributivity is decidable in the case of one variable.
\end{prop}

\begin{exam}[\bf syntactic solution]\label{X:Syntactic}
Consider $\TT:= (\xx \op \xx) \op (\xx \op (\xx \op \xx))$ and $\TT' := \xx \op ((\xx \op \xx) \op (\xx \op \xx))$. Expanding~$\TT$ at~$(1, \emptyset)$ and~$\TT'$ at~$(\emptyset, 1, 0)$ (see Definition~\ref{D:Blueprint} for the formalism), we obtain the common LD-expansion
$$((\xx \op \xx) \op (\xx \op \xx)) \op ((\xx \op \xx) \op (\xx \op \xx)),$$
and we conclude that $\TT$ and~$\TT'$ are $\LD$-equivalent. Note that, here, the conclusion is certain without any hypothesis (no need of an acyclic left-shelf), contrary to the case of expansions that are proper iterated left subterms of one another.
\end{exam}

\subsection{A syntactic solution}\label{SS:Syntactic}

When the above method was first described (1989), no example of an acyclic left-shelf was known and its existence was a conjecture. Shortly after, R.\,Laver established in~\cite{Lvb}: 

\begin{prop}[\bf acyclic I]\label{P:Acyclic1}
If $\jj$ is a nontrivial elementary embedding of~$\VV_\lambda$ into itself, the left-shelf~$\Iter(\jj)$ is acyclic.
\end{prop}

This resulted in the paradoxical situation of a finitistic problem (the word problem of~$\LD$) whose only known solution appeals to an unprovable axiom: 

\begin{coro}\label{C:Correct1}
If there exists a Laver cardinal, Algorithm~\ref{A:Syntact} is correct, and the word problem of selfdistributivity is decidable in the case of one variable.
\end{coro}

The previous puzzling situation was resolved by the construction, without any set theoretical assumption, of another acyclic left-shelf~\cite{Dez, Dfb}:

\begin{prop}[\bf acyclic II]\label{P:Acyclic2}
The braid shelf of Example~\ref{X:Braids} is acyclic.
\end{prop}

\begin{proof}[Idea of the proof]
As we are considering left selfdistributivity here, the relevant version is the operation~$\op$ defined on the braid group~$\Bi$ by
\begin{equation}\label{E:BraidLD}
\aa \op \bb:= \aa \, \sh(\bb) \, \sig1 \, \sh(\aa)\inv.
\end{equation}
Proving that $(\Bi, \op)$ is acyclic means that no equality of the form 
\begin{equation}\label{E:Acyclic1}
\aa = (\pdots ((\aa \op \bb_1) \op \bb_2) \pdots ) \op \bb_\nn
\end{equation}
with $\nn \ge 1$ is possible in~$\Bi$. According to the definition of the operation~$\op$, the right term in~\eqref{E:Acyclic1} expands into an expression of the form
\begin{equation}\label{E:Acyclic2}
\aa \cdot \sh(\cc_0) \sig1 \sh(\cc_1) \sig1 \pdots \sig1 \sh(\cc_\nn),
\end{equation}
and, therefore, for excluding \eqref{E:Acyclic1}, it suffices to prove that a braid of the form
\begin{equation}\label{E:Sig1Pos}
\sh(\cc_0) \sig1 \sh(\cc_1) \sig1 \pdots \sig1 \sh(\cc_\nn)
\end{equation}
is never trivial (equal to~$1$). It is natural to call the braids as in~\eqref{E:Sig1Pos} \emph{$\sig1$-positive}, since they admit a decomposition, in which there is at least one letter~$\sig1$ and no letter~$\siginv1$. So, the problem is to show that a $\sig1$-positive braid is never trivial. 

Several arguments exist, see in particular~\cite{Dfb}, the simplest being the one, due to D.\,Larue~\cite{Lra}, which appeals to the Artin representation of~$\Bi$ in~$\Aut(\Fi)$, where $\Fi$ denotes a free group based on an infinite family $\{\xx_\ii \mid \ii \ge 1\}$, identified with the family of all freely-reduced words on~$\{\xx_\ii^{\pm1} \mid \ii \ge 1\}$. Artin's representation is defined by the rules
\begin{equation*}
\rho(\sig\ii)(\xx_\ii) := \xx_\ii \xx_{\ii + 1} \xx_\ii\inv, \quad \rho(\sig\ii)(\xx_{\ii + 1}) := \xx_\ii, \quad \rho(\sig\ii)(\xx_\kk) := \xx_\kk \text{ for $\kk \not= \ii, \ii + 1$},
\end{equation*}
and simple arguments about free reduction show that, if $\cc$ is a $\sig1$-positive braid, then $\rho(\cc)$ maps~$\xx_1$ to a reduced word that finishes with the letter~$\xx_1\inv$ and, therefore, $\cc$ cannot be trivial, since $\rho(1)$ maps~$\xx_1$ to~$\xx_1$, which does not finish with~$\xx_1\inv$. 
\end{proof}

Applying Prop.~\ref{P:CondSol}, we remove the exotic assumption in  Corollary~\ref{C:Correct1}:

\begin{coro}
Algorithm~\ref{A:Syntact} is correct, and the word problem of selfdistributivity is decidable in the case of one variable.
\end{coro}

\subsection{A semantic solution}\label{SS:Semantic}

However, we can obtain more, namely a new, more efficient algorithm for the word problem of~LD. The starting point is the following criterion, whose proof is essentially the same as the one of Prop.~\ref{P:CondSol}:

\begin{lemm}[\bf freeness criterion]\label{L:Freeness}
If $\SS$ is an acyclic monogenerated left-shelf, then $\SS$ is free.
\end{lemm}

\begin{proof}
Assume that $\SS$ is generated by~$\gg$. Let $\TT, \TT'$ be two terms in~$\TERM\op\xx$. As above, we write $\TT(\gg)$ for the evaluation of~$\TT$ at $\xx:= \gg$. As $\SS$ is a left-shelf, $\TT \eqLD \TT'$ certainly implies $\TT(\gg) = \TT'(\gg)$. Conversely, assume $\TT \not\eqLD \TT'$. By Lemma~\ref{L:Comp2}, at least one of $\TT \divsLD \TT'$, $\TT \multsLD \TT'$ holds, say for instance $\TT \divsLD \TT'$. By projection, we deduce $\TT(\gg) \divs \TT'(\gg)$ in~$\SS$, whence $\TT(\gg) \not=\TT'(\gg)$, since $\div$ has no cycle in~$\SS$. Therefore, $\TT \eqLD \TT'$ is equivalent to $\TT(\gg) = \TT'(\gg)$, and $\SS$ is free.  
\end{proof}

We deduce

\begin{prop}[\bf realization]
For every braid~$\aa$, the sub-left-shelf of~$(\Bi, \op)$ generated by~$\aa$ is free.
\end{prop}

This applies in particular for $\aa = 1$; the braids obtained from~$1$ by iterating~$\op$ are called \emph{special} in~\cite{Dgd}, so special braids provide a realization~$\BRsp$ of the rank~$1$ free left-shelf~$\Free_1$. Efficient solutions of the word problem for the presentation~\eqref{E:BraidPres} of~$\Bi$ are known~\cite{Eps, Dfo}, \ie, algorithms that decide whether or not a word in the letters~$\sigg\ii{\pm1}$ represents~$1$ in~$\Bi$. We deduce a simple semantic algorithm for the word problem of~$\LD$. Below, we use~$\BWi$ for the family of all braid words, $\EQUIV$ for a solution of the word problem of~\eqref{E:BraidPres}, $\concat$ for word concatenation, $\SHIFT$ for braid word shifting, and $\INV$ for braid word formal inversion.

\begin{algo}[\bf semantic solution, case of one variable]\label{A:Semant}

\rm\hfill\par

\noindent {\bf Input}: Two terms $\TT$, $\TT'$ in~$\TERM\op\xx$

\noindent {\bf Output}: $\True$ if $\TT$ and $\TT'$ are $\LD$-equivalent, $\False$ otherwise

\lab{1} $\ww:= \EVAL(\TT)$

\lab{2} $\ww':= \EVAL(\TT')$

\lab{3} {\bf return} $\EQUIV(\ww, \ww')$

\medskip

\lab{7} {\bf function} $\EVAL(\TT:$ term): word in~$\BWi$

\lab{8} {\bf if} $\TT = \xx \in \XX$ {\bf then}

\lab{9} \quad {\bf return} $\ew$

\lab{10} {\bf else if} $\TT = \TT_0 \op \TT_1$ {\bf then}

\lab{11} \quad {\bf return} $\EVAL(\TT) \concat \SHIFT(\EVAL(\TT')) \concat \sig1 \concat \INV(\SHIFT(\EVAL(\TT)))$

\end{algo}

The inductive definition of~$\op$ implies that the braid word associated with a term~$\TT$ of size~$\nn$ has length at most~$2^{O(\nn)}$. On the other hand, in connection with the existence of an automatic structure on braid groups, there exist solutions of the word problem for~\eqref{E:BraidPres} of quadratic complexity~\cite{Eps}, so the overall complexity of Algorithm~\ref{A:Semant} is simply exponential---whereas the only proved upper bound for the complexity of Algorithm~\ref{A:Syntact} is a tower of exponentials of exponential height.

\begin{exam}[\bf semantic solution]\label{X:Semantic}
Consider $\TT:= (\xx \op \xx) \op (\xx \op (\xx \op \xx))$ and $\TT' := \xx \op ((\xx \op \xx) \op (\xx \op \xx))$. The reader is invited to check
$$\EVAL(\TT) = \sig1 \sig3 \sig2 \sig1 \siginv2, \qquad \EVAL(\TT') = \sig2 \sig3 \sig2 \siginv3 \sig1.$$
As the latter braid words are equivalent (they both represent~$\sig3 \sig2 \sig1$ in~$\Bi$), we conclude that $\TT$ and~$\TT'$ are $\LD$-equivalent.
\end{exam}

\section{Word problem, the case of shelves II}\label{S:WPLD2}

We thus obtained in Section~\ref{S:WPLD1} a positive solution for the word problem of selfdistributivity in the case of terms in one variable. At this point, several questions remain open: Where does the exotic braid operation of~\eqref{E:BraidShelf} or~\eqref{E:BraidLD} come from? What about the case of terms involving more than one variable? Can one find solutions based on normal terms, that is, describe a distinguished representative in every $\LD$-class? Are there efficient syntactic solutions (the one of Section~\ref{S:WPLD1} is not)? These four questions are addressed in the four subsections below.

\subsection{Where does the braid shelf come from?}\label{SS:Thompson}

The answer lies in the approach developed in~\cite{Dfb}, which is parallel to the treatment of associativity and commutativity by R.\,Thompson in the 1970s~\cite{McT, Tho, CFP}. In addition to the braid application, one also obtains a direct proof that the free left-shelf~$\Free_1$ is acyclic, without refering to any concrete realization like the one based on braids or the one based on iterations of elementary embeddings. 

The idea is to see the relations~$\eqLD$ and~$\expLD$ as the result of applying the action of a monoid on terms. To this end, we take into account the positions and the orientations, where the selfdistributivity law is applied, as already alluded to in the definition of the procedure~$\EXPAND$ of Algorithm~\ref{A:Syntact}. Every term~$\TT$ in~$\TERM\op\XX$ that is not a variable (\ie, an element of~$\XX$) admits a left and a right subterm. Iterating, we can specify each subterm of~$\TT$ by a finite sequence of~0s (for left) and~1s (for right): such finite sequences will be called \emph{addresses}, denoted~$\alpha, \beta$,..., and we use $\sub\TT\alpha$ for the \emph{$\alpha$-subterm} of~$\TT$, that is, the subterm of~$\TT$ that corresponds to the fragment below the address~$\alpha$ in the tree associated with~$\TT$. With this notation, $\sub\TT0$ (\resp $\sub\TT1$) is the left (\resp right) subterm of~$\TT$. Note that $\sub\TT\alpha$ exists only for~$\alpha$ short enough (the family of all~$\alpha$s for which $\sub\TT\alpha$ exists will be called the \emph{skeleton} of~$\TT$).

\begin{defi}[{\bf operator}~$\LDop\alpha$, {\bf monoid}~$\Geom\LD$]
For each address~$\alpha$, we denote by~$\LDop\alpha$ the partial operator on~$\TERM\op\XX$ such that $\TT \act \LDop\alpha$ is defined if $\sub\TT\alpha$ exists and can be written as $\TT_1 \op (\TT_2 \op \TT_3)$, in which case $\TT \act \LDop\alpha$ is the term obtained by replacing the latter subterm with $(\TT_1 \op \TT_2) \op (\TT_1 \op \TT_3)$. The \emph{geometry monoid} of~$\LD$ is the monoid~$\Geom\LD$ generated by all operators~$\LDop\alpha$ and their inverses under composition
\end{defi}

Thus applying~$\LDop\alpha$ means applying the LD law at position~$\alpha$ in the expanding direction. By definition, two terms~$\TT, \TT'$ are LD-equivalent if, and only if, some element of the monoid~$\Geom\LD$ maps~$\TT$ to~$\TT'$. When the selfdistributivity law~$\LD$ is replaced with the associativity law~$\Ass$, the corresponding geometry monoid~$\Geom\Ass$ turns out to (essentially) Richard Thompson's group~$F$~\cite{Tho}. It is important to note that the action of~$\Geom\LD$ is partial: for instance, $\TT \act \LDop\alpha$ is defined only when~$\alpha$ is short enough (precisely: when $\alpha0$, $\alpha10$, and $\alpha11$ belong to the skeleton of~$\TT$). 

For the sequel, it is crucial to work in a group context. However, $\Geom\LD$ is not a group, but only an inverse monoid: exchanging~$\LDop\alpha$ and~$\LDop\alpha\inv$ and reversing the order of factors provides for every~$\gg$ in~$\Geom\LD$ an element~$\gg\inv$ satisfying $\gg \gg\inv \gg = \gg$ and $\gg\inv \gg \gg\inv = \gg\inv$, but $\gg\gg\inv$ is only the identity of its domain. Contrary to the case of associativity, no quotient-group of~$\Geom\LD$ is useful. However, one can \emph{guess} a list of relations~$\Rel\LD$ that connect the maps~$\LD_\alpha$ and consider the abstract group~$\Geomt\LD$ presented by~$\Rel\LD$: if $\Rel\LD$ is exhaustive enough, we can hope to work with~$\Geomt\LD$ as we did with~$\Geom\LD$. 

Here is the key point. The absorption property (Lemma~\ref{L:Absorp}) implies that, for every~$\TT$ in~$\TERM\op\xx$, the terms~$\xx^{[\nn + 1]}$ and~$\TT \op \xx^{[\nn]}$ are LD-equivalent for~$\nn$ large enough. Hence, there exists an element in~$\Geom\LD$ that maps~$\xx^{[\nn + 1]}$ to~$\TT \op \xx^{[\nn]}$. Reading step by step the inductive proof of Lemma~\ref{L:Absorp} shows that such an element can be defined inductively as follows (note that $\nn$ does not occur): 

\begin{defi}[\bf blueprint]\label{D:Blueprint}
For~$\TT$ in~$\TERM\op\xx$, we inductively define~$\chi(\TT)$ in~$\Geom\LD$ by
\begin{equation}\label{E:Blueprint}
\chi(\TT):= 
\begin{cases}
1&\text{for $\TT = \xx$},\cr
\chi(\TT_0) \cdot \sh_1(\chi(\TT_1)) \cdot \LDop\emptyset \cdot \sh_1(\chi(\TT_1))\inv&\text{for $\TT = \TT_0 \op \TT_1$},
\end{cases}
\end{equation}
\end{defi}

We denote by~$\chit(\TT)$ the element of the group~$\Geomt\LD$ defined by a similar induction. Then $\chit(\TT)$ should be viewed as a sort of copy of~$\TT$ inside~$\Geomt\LD$ (the ``\emph{blueprint}'' of~$\TT$). We then can obtain a non-conditional proof of the exclusion property (Lemma~\ref{L:Excl1}):

\begin{lemm}[\bf exclusion property]\label{L:Excl2}
The relations $\eqLD$ and~$\divsLD$ on~$\TERM\op\xx$ exclude one another.
\end{lemm}

\begin{proof}[Sketch of the proof]
Assume $\TT \eqLD \TT'$. There exists~$\gg$ in~$\Geom\LD$ that maps~$\TT$ to~$\TT'$, and therefore~$\sh_0(\gg)$ maps~$\TT \op \xx^{[\nn]}$ to~$\TT' \op \xx^{\nn}$, where $\sh_0(\gg)$ is the shifted version of~$\gg$ that consists in applying~$\gg$ in the left subterm (that is, replacing~$\LDop\alpha$ with~$\LDop{0\alpha}$ everywhere in~$\gg$). Then both $\chi(\TT')$ and $\chi(\TT) \sh_0(\gg)$ map~$\xx^{[\nn + 1]}$ to~$\TT' \op \xx^{[\nn]}$, so the quotient $\chi(\TT)\inv \chi(\TT')$ belongs to~$\sh_0(\Geom\LD)$. Using the explicit presentation of the group~$\Geomt\LD$, one checks that, similarly, the quotient $\chit(\TT)\inv \chit(\TT')$ belongs to~$\sh_0(\Geomt\LD)$. 

Assume now $\TT \divsLD \TT'$. Then one reduces to the case $\TT \divs \TT'$, in which case $\chit(\TT)\inv \chit(\TT')$ has an expression in which the generator~$\LDop\emptyset$ appears but its inverse does not. An algebraic study of the group~$\Geomt\LD$ as group of right fractions then enables one to prove that $\chit(\TT)\inv \chit(\TT')$ does not belong to~$\sh_0(\Geomt\LD)$---this is the key point.
\end{proof}

As a first consequence, one directly deduces:

\begin{prop}[\bf acyclic III]\label{P:Acyclic3}
The free left-shelf~$\Free_1$ is acyclic.
\end{prop}

In the context of Prop.~\ref{P:CondSol}, this gives another proof of the validity of Algorithm~\ref{A:Syntact}, hence of the solvability of the word problem of~$\LD$---the first complete one chronologically~\cite{Dez}. We also obtain another solution using the group~$\Geomt\LD$. Indeed, Lemma~\ref{L:Excl2} shows that $\TT \eqLD \TT'$ holds if, and only if, $\chit(\TT)\inv \chit(\TT')$ belongs to the subgroup~$\sh_0(\Geomt\LD)$, which can be tested effectively (we skip the description).

\begin{exam}[\bf blueprint solution]\label{X:Solution}
Consider $\TT:= (\xx \op \xx) \op (\xx \op (\xx \op \xx))$ and $\TT' := \xx \op ((\xx \op \xx) \op (\xx \op \xx))$ again. Then one finds (we write $\alpha$ for~$\LDop\alpha$):
$$\chit(\TT)= \LDop\emptyset \, \LDop{11} \, \LDop{1} \, \LDop\emptyset \, \LDop1\inv, \qquad \chit(\TT')= \LDop{1} \, \LDop{11} \, \LDop{1} \, \LDop{11}\inv \, \LDop1,$$
and we can check $\chit(\TT)\inv \chit(\TT') = \LDop{01} \LDop{0} \LDop{01}\inv \LDop{00}\inv \LDop0\inv$: the latter element of the group~$\Geomt\LD$ belongs to the image of~$\sh_0$, hence $\TT'$ and~$\TT'$ are $\LD$-equivalent.
\end{exam}

A second consequence is the construction of the braid operation of~\eqref{E:BraidLD}. The explicit form of the relations~$\Rel\LD$---which we did not mention so far---implies that Artin's braid group~$\Bi$ is a quotient of the geometry group~$\Geomt\LD$, namely the one obtained when all generators~$\LDop\alpha$ such that $\alpha$ contains at least one~$0$ are collapsed. Indeed, the remaining relations take the form
\begin{gather*}
\LDop{1^\ii} \, \LDop{1^\jj} \, \LDop{1^\ii} = \LDop{1^\jj} \, \LDop{1^\ii} \,\LDop{1^\ii} \, \LDop{1^\ii0} \qquad\text{for $\jj = \ii + 1 \ge 1$},\\ 
\LDop{1^\ii} \, \LDop{1^\jj} = \LDop{1^\jj} \, \LDop{1^\ii} \qquad\text{for $\jj \ge \ii + 2 \ge 2$,}
\end{gather*}
which project to the relations of~\eqref{E:BraidPres} when $\LDop{1^\ii0}$ is collapsed and $\LDop{1^\ii}$ is mapped to~$\sig{\ii+1}$. We saw that $\TT \eqLD \TT'$ holds if, and only if, the quotient $\chit(\TT)\inv \chit(\TT')$ belongs to the subgroup~$\sh_0(\Geomt\LD)$. Hence, when all generators~$\LDop\alpha$ with~$0$ in~$\alpha$ are collapsed, $\chit(\TT)\inv \chit(\TT')$ goes to~$1$, that is, $\chit(\TT)$ and~$\chit(\TT')$ have the same image. In other words, if we consider the operation on the quotient that mimicks the inductive definition of~\eqref{E:Blueprint}, then that operation must obey the $\LD$-law. The reader is invited to check that the operation introduced in this way is precisely that of~\eqref{E:BraidLD}. Therefore the latter does not appear out of the blue, but it directly stems from~\eqref{E:Blueprint}, which itself follows the inductive proof of the absorption property of Lemma~\ref{L:Absorp}.

We conclude with two remarks. First, the relations of~$\Rel\LD$ can be stated so as to involve no~$\LDop\alpha\inv$ and, therefore, one can introduce the monoid~$\Geomtp\LD$ they present. The main open question in this area is 

\begin{ques}[\bf embedding conjecture]\label{Q:Emb}
Does the monoid~$\Geomtp\LD$ embeds in the group~$\Geomt\LD$?
\end{ques}

A positive answer is conjectured. It amounts to proving that $\Geomtp\LD$, which is left cancellative, is also right cancellative, and it would imply a number of structural properties for selfdistributivity~\cite[Chapter~IX]{Dgd}, in particular that the LD-expansion relation~$\expLD$ admits least upper bounds, thus reminiscent of associahedra and Tamari lattices for associativity~\cite{Mul}.

The second remark is that a similar approach can be developed for other algebraic laws, for instance for the ``central duplication'' law $\xx(\yy\zz) = (\xx\yy)(\yy\zz)$, where it provides the only known solution of the word problem~\cite{Dgj}.

\subsection{Normal forms}\label{SS:NF}

Connected with the word problem of selfdistributivity is the question of finding normal forms, namely finding one distinguished term in every $\LD$-equivalence class. Several solutions have been described, in particular by R.\,Laver in~\cite{Lvb} and in subsequent works~\cite{Lvc, LvM}. Here we sketch the solution developed in~\cite{Dfd}, which is simpler and directly follows from the properties mentioned above. Once again, we concentrate on the case of one variable.

Let $\TT$ belong to~$\TERM\op\xx$. By the absorption property (Lemma~\ref{L:Absorp}), the term~$\TT \op \xx^{[\nn]}$ is $\LD$-equivalent to~$\xx^{[\nn + 1]}$ for~$\nn$ large enough, hence, by the confluence property (Lemma~\ref{L:Confl}), they admit a common LD-expansion, which, by the proof of Lemma~\ref{L:Confl}, may be assumed to be of the form~$\partial^\pp\xx^{[\nn]}$ for some~$\pp$. Then, as explained in the proof of the comparison property (see Fig.~\ref{F:Comp}), there exists a number~$\rr$ such that the iterated left subterm~$\LS^\rr(\partial^\pp\xx^{[\nn]})$ is $\LD$-equivalent to~$\TT$, and, for given~$\pp$ and~$\nn$, this number~$\rr$ is necessarily unique by the exclusion property. As is usual in this context, the index~$\nn$ does not matter (provided it is large enough) and, so, by choosing~$\pp$ to be minimal, we obtain a distinguished representative:

\begin{prop}[\bf normal form]\label{P:Normal}
Call a term~$\TT$ \emph{normal} if $\TT$ is an iterated left subterm of~$\partial^\pp\xx^{[\nn]}$ for some~$\pp$ and~$\nn$, and $\pp$ is minimal with that property. Then every term in~$\TERM\op\xx$ is $\LD$-equivalent to a unique normal term.
\end{prop}

The above construction is effective, but it does not give an explicit description of normal terms. We provide it now. To this end, it is convenient to extend the notion of an iterated left subterm into that of a \emph{cut}. By definition, an iterated left subterm of~$\TT$ corresponds to extracting from~$\TT$ the fragment that lies under some address~$0^\rr$, hence on the left of some leaf with address~$0^\rr 1^\ss$. We extend the definition to all fragments corresponding to any leaf in (the tree associated with)~$\TT$.

The easiest way to state a formal definition is to start from the \emph{(right) Polish} expression of~$\TT$: 

\begin{defi}[\bf Polish expression]\label{D:Polish}
For~$\TT$ in~$\TERM\op\XX$, we denote by~$\Pol\TT$ the word in the letters~$\xx$ and~$\opp$ inductively defined by~$\Pol\xx := \xx$ and $\Pol\TT:= \Pol{\TT_0} {\concat} \Pol{\TT_1} {\concat} \opp$ for~$\TT = \TT_0 \op \TT_1$ (using $\concat$ for concatenation). 
\end{defi}

The order of symbols in~$\Pol\TT$ corresponds to the left--right--root enumeration of the associated binary tree, and one obtains a one-to-one correspondence between the nodes in the tree (associated with)~$\TT$, hence what we called the skeleton of~$\TT$, and the letters of the word~$\Pol\TT$. We denote by~$\add(\pp, \TT)$ the address that corresponds to the $\pp$th letter in~$\Pol\TT$.

\noindent\begin{minipage}{\textwidth}\rightskip25mm
\hspace{19pt}For instance, \VR(3,0)for $\TT = (\xx \op (\xx \op \xx)) \op \xx$ (see on the right), one finds $\Pol\TT = \xx\xx\xx{\opp}{\opp}\xx{\opp}$, and the correspondence between letters and addresses is \VR(0,1) as follows:

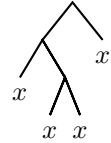
\hfill\begin{tabular}{c|p{5.5mm}p{5.5mm}p{5.5mm}p{5.5mm}p{5.5mm} p{5.5mm} p{5.5mm} p{5.5mm}}
$\pp$ & \hfil1 & \hfil2 & \hfil3 & \hfil4 & \hfil5 & \hfil6 & \hfil7\\
\hline
$\pp$th letter of~$\Pol\TT$\VR(3.5,1.5) &\hfil$\xx$ &\hfil$\xx$ &\hfil$\xx$ &\hfil$\opp$ &\hfil$\opp$ &\hfil$\xx$ &\hfil$\opp$\\
\hline
$\add(\pp, \TT)$ &\hfil$00$ &\hfil{$010$} &\hfil$011$&\hfil$01$ &\hfil$0$ &\hfil$1$ &\hfil$\emptyset$
\end{tabular}\hfill
\begin{picture}(0,0)(-10,0)
\psline(0,5)(3,10)(6,5)(4,0)(6,5)(8,0)(6,5)(3,10)(7,15)(11,10)
\put(-1,2){$\xx$}
\put(3,-3){$\xx$}
\put(7,-3){$\xx$}
\put(10,7){$\xx$}
\end{picture}
\end{minipage}

A \VR(3.5,0) word~$\ww$ in the letters~$\{\xx, \opp\}$ is the Polish expression of a term if, and only if, the number~$\vert\ww\vert_\XX$ of~variables in~$\ww$ is one more than the number~$\vert\ww\vert_{\opp}$ of~$\opp$ and, for every nonempty initial factor~$\ww'$ of~$\ww$, one has $\vert\ww'\vert_\XX > \vert\ww'\vert_{\opp}$. Iterated left subtrees then correspond to prefixes~$\ww'$ of~$\Pol\TT$ that themselves are Polish expressions, \ie, satisfy $\vert\ww'\vert_\XX = \vert\ww'\vert_{\opp} + 1$. 

For every~$\alpha$ that is the address of a leaf in~$\TT$, hence corresponds to a variable~$\xx_\ii$ in~$\Pol\TT$, if $\ww'$ is the prefix of~$\Pol\TT$ that finishes with this letter~$\xx_\ii$, there exists a unique nonnegative number~$\pp$ such that $\ww' {\concat} {\opp}^\pp$ is a Polish expression: the number~$\pp$ is the local defect in letters~$\opp$, and it is zero if, and only if, $\ww'$ is the Polish expression of an iterated left subterm of~$\TT$, hence if, and only if, $\alpha$ is of the form~$0^\rr 1^\ss$.

\begin{defi}[\bf cut]
For every term~$\TT$ and every address~$\alpha$ of a leaf of~$\TT$, we let~$\cut(\TT, \alpha)$ be the term, whose Polish expression is the word~$\ww' {\concat} {\opp}^\pp$ as above.
\end{defi}

\begin{exam}[\bf cut]
Let again $\TT$ be $(\xx \op (\xx \op \xx)) \op \xx$. The size of~$\TT$ (number of occurrences of variables) is~$4$, so there are four~leaves, at addresses~$00$, $010$, $011$, and $1$, and the corresponding four~cuts are
$$\cut(\TT, 00) = \xx, \ \cut(\TT, 010) = \xx \op \xx, \ \cut(\TT, 011) = \xx \op (\xx \op \xx), \ \cut(\TT, 1) = \TT.$$ 
\end{exam}

Cuts can alternatively be defined without mentioning the Polish expressions as follows: if a binary address~$\alpha$ contains $\mm$~times~$1$, it has the form $0^{\rr_0} 1 0^{\rr_1} 1 \pdots 1 0^{\rr_\mm}$, and, then, when $\alpha$ is the address of a leaf in~$\TT$, one has $$\cut(\TT, \alpha) = \sub\TT{\alpha_0} \op (\sub\TT{\alpha_1} \op \pdots (\sub\TT{\alpha_{\mm - 1}} \op \sub\TT{\alpha_\mm}) \pdots )),$$
where we put $\alpha_\kk:= 0^{\rr_0} 1 0^{\rr_1} 1 \pdots 1 0^{\rr_\kk + 1}$ for $0 \le \kk < \mm$, and $\alpha_\mm := \alpha$.

The main result now is that the cuts of the term~$\partial\TT$ of~\eqref{E:Der} can be simply described from those of~$\TT$, inductively leading to a description of normal terms.

\begin{defi}[\bf descent]
If $\alpha, \beta$ are binary addresses, declare $\alpha \ggg \beta$ if there exists an adress~$\beta$ satisfying $\alpha = \gamma 1^\pp$ and $\beta = \gamma 0 \delta$ for some~$\pp \ge 1$ and~$\delta$. Define a \emph{descent} in~$\TT$ to be a finite sequence of addresses $(\alpha_1 \wdots \alpha_\mm)$ such that, for every~$\ii$, the address~$\alpha_\ii$ is that of a leaf in~$\TT$ and $\alpha_\ii \ggg \alpha_{\ii + 1}$ holds for~$\ii < \mm$. 
\end{defi}

\begin{lemm}
There exists a unique bijection~$\phi$ between the leaves of~$\partial\TT$ and the descents of~$\TT$ such that $\phi(\alpha) = (\alpha_1 \wdots \alpha_\mm)$ implies
$$\cut(\partial\TT, \alpha) \eqLD \cut(\TT, \alpha_1) \op (\cut(\TT, \alpha_2) \, \op \pdots (\cut(\TT,{\alpha_{\mm - 1})} \op \cut(\TT, \alpha_\mm)) \pdots)).$$
\end{lemm}

So the cuts of~$\partial\TT$, hence in particular its iterated left subterms, are simply defined in terms of those of~$\TT$. In this way, we inductively obtain a description of all normal terms in terms of nested sequences of leaf addresses in~$\xx^{[\nn]}$ and the associated cuts, which are the terms~$\xx^{[\ii]}$ with $\ii < \nn$. In other words, we obtain a unique distinguished expression for every term in terms of~$\xx^{[\ii]}$. Say that a normal term has \emph{degree~$\pp$} if it occurs as a cut of~$\partial^\pp \xx^{[\nn]}$ and $\pp$ is minimal with that property.

\begin{exam}[\bf normal terms] (See Fig.~\ref{F:Normal}.)
Here we describe all normal terms of degree at most~$3$ below~$\xx^{[3]}$. For degree~$0$, there are two proper cuts of~$\xx^{[3]}$, namely~$\xx^{[1]}$ (that is, $\xx$) and~$\xx^{[2]}$, hereafter abridged as~$1$ and~$2$.

For degree~$1$, excluding~$(11)$, there are three descents in~$\xx^{[3]}$, namely $(0)$, $(10)$, and $(10,0)$, leading to three proper cuts of~$\partial\xx^{[3]}$, namely $\xx^{[1]}$, $\xx^{[2]}$, and~$\xx^{[2]}\op\xx^{[1]}$, or $1$, $2$, and $21\opp$ in abridged Polish notation. As $1$ and $2$ already appeared as cuts of~$\xx^{[3]}$, there is only one normal term of degree~$1$, namely~$21\opp$.

For degree~$2$, excluding~$(11)$, there are five descents in~$\partial\xx^{[3]}$, namely $(00)$, $(01)$, $(10)$, $(10,00)$, and $(10, 01)$, leading to five proper cuts of~$\partial^2\xx^{[3]}$, namely, in abridged notation, $1$, $2$, $21\opp$, $21\opp1\opp$, and $21\opp2\opp$. As $1$, $2$, and $21\opp$ already appeared as cuts of~$\partial\xx^{[3]}$, there are two normal terms of degree~$2$, namely~$21\opp1\opp$ and $21\opp2\opp$.

For degree~$3$, excluding~$(11)$, there are eleven descents in~$\partial^2\xx^{[3]}$, namely $(000)$, $(001)$, $(01)$, $(100)$, $(100,000)$, $(100, 001)$, $(100, 01)$, $(101)$,  $(101, 000)$, $(101, 001)$, and $(101, 01)$, leading to eleven proper cuts of~$\partial^3\xx^{[3]}$, namely $1$, $2$, $21\opp$, $21\opp1\opp$, $21\opp1\opp1\opp$, $21\opp1\opp2\opp$
, $21\opp1\opp21\opp\opp$, $21\opp2\opp$, $21\opp2\opp1\opp$, $21\opp2\opp2\opp$, and $21\opp2\opp21\opp\opp$. As $1$, $2$, $21\opp$, $21\opp1\opp$, and $21\opp2\opp$ already appeared as cuts of~$\partial^2\xx^{[3]}$, there are six normal terms of degree~$3$, namely~$21\opp1\opp1\opp$, $21\opp1\opp2\opp$, $21\opp1\opp21\opp\opp$, $21\opp2\opp1\opp$, $21\opp2\opp2\opp$, and $21\opp2\opp21\opp\opp$. 
\end{exam}

\begin{figure}[htb]
\begin{picture}(115,36)(0,-9)
\psset{unit=0.9mm}
\setlength\unitlength{0.9mm}
\put(0,0){\begin{picture}(15,25)(0,0)
\psline(0,20)(3,25)(6,20)(9,15)(6,20)(3,15)
\put(-1.5,17){\begin{picture}(13,5)
\psframe[linewidth=0.8pt,framearc=.5](-1,-1.2)(3,3.2)
\put(0,0){\small$1$}
\end{picture}}
\put(1.5,12){\begin{picture}(13,5)
\psframe[linewidth=0.8pt,framearc=.5](-1,-1.2)(3.2,3.2)
\put(0,0){\small$2$}
\end{picture}}
\put(7,12){$\pdots$}
\put(2,26.5){$\xx^{[3]}$}
\end{picture}}
\put(18,0){\begin{picture}(15,25)(0,0)
\psline(0,15)(2,20)(4,15)(2,20)(6,25)(10,20)(8,15)(10,20)(12,15)
\put(6.5,12.2){\rotatebox{-45}{\begin{picture}(13,5)
\psframe[linewidth=0.8pt,framearc=.5](-1,-1.2)(6.5,3.2)
\put(0,0){\small$21\opp$}
\end{picture}}}
\put(11,12){$\pdots$}
\put(-1,12){\small$1$}
\put(3,12){\small$2$}
\put(4,26.5){$\partial\xx^{[3]}$}
\end{picture}}
\put(37,0){\begin{picture}(28,25)(0,0)
\psline(0,10)(3,15)(6,10)(3,15)(7.5,20)(12,15)(7.5,20)(15,25)(22.5,20)(27,15)(22.5,20)(18,15)(21,10)(18,15)(15,10)
\put(13,8){\rotatebox{-45}{\begin{picture}(20,5)
\psframe[linewidth=0.8pt,framearc=.5](-1,-1.2)(10,3.2)
\put(0,0){\small$21\opp1\opp$}
\end{picture}}}
\put(20,8){\rotatebox{-45}{\begin{picture}(20,5)
\psframe[linewidth=0.8pt,framearc=.5](-1,-1.2)(10,3.2)
\put(0,0){\small$21\opp2\opp$}
\end{picture}}}
\put(25,12){$\pdots$}
\put(-1,7){\small$1$}
\put(5,7){\small$2$}
\put(10,12){\small$21\opp$}
\put(12,26.5){$\partial^2\xx^{[3]}$}
\end{picture}}
\put(71,0){\begin{picture}(50,25)(0,0)
\psline(0,0)(2,5)(4,0)(2,5)(5,10)(8,5)(5,10)(9,15)(13,10)(9,15)(17,20)(25,15)(29,10)(25,15)(21,10)(24,5)(21,10)(18,5)(16,0)(18,5)(20,0)
\put(9,7){\small$21\opp1\opp$}
\put(25,7){\small$21\opp2\opp$}
\psline(17,20)(33,25)(49,20)(57,15)(49,20)(41,15)(45,10)(41,15)(37,10)(40,5)(37,10)(34,5)
\put(13.5,-2){\rotatebox{-45}{\begin{picture}(20,5)
\psframe[linewidth=0.8pt,framearc=.5](-1,-1.2)(13.5,3.2)
\put(0,0){\small$21\opp1\opp1\opp$}
\end{picture}}}

\put(19.5,-2){\rotatebox{-45}{\begin{picture}(20,5)
\psframe[linewidth=0.8pt,framearc=.5](-1,-1.2)(13.5,3.2)
\put(0,0){\small$21\opp1\opp2\opp$}
\end{picture}}}

\put(23,3){\rotatebox{-45}{\begin{picture}(20,5)
\psframe[linewidth=0.8pt,framearc=.5](-1,-1.2)(17,3.2)
\put(0,0){\small$21\opp1\opp21\opp\opp$}
\end{picture}}}
\put(32,3){\rotatebox{-45}{\begin{picture}(20,5)
\psframe[linewidth=0.8pt,framearc=.5](-1,-1.2)(13.5,3.2)
\put(0,0){\small$21\opp2\opp1\opp$}
\end{picture}}}
\put(39.5,3){\rotatebox{-45}{\begin{picture}(20,5)
\psframe[linewidth=0.8pt,framearc=.5](-1,-1.2)(13.5,3.2)
\put(0,0){\small$21\opp2\opp2\opp$}
\end{picture}}}
\put(44.5,8){\rotatebox{-45}{\begin{picture}(20,5)
\psframe[linewidth=0.8pt,framearc=.5](-1,-1.2)(17,3.2)
\put(0,0){\small$21\opp2\opp21\opp\opp$}
\end{picture}}}
\put(55,12){$\pdots$}
\put(-1,-3){\small$1$}
\put(3,-3){\small$2$}
\put(6,2){\small$21\opp$}
\put(30,26.5){$\partial^3\xx^{[3]}$}
\end{picture}}
\end{picture}
\caption[]{\sf The terms~$\partial^\pp\xx^{[3]}$ for $0 \le \pp \le 3$, and the canonical decompositions of their cuts: for each degree, the cuts that are normal, that is, that do not appear in a lower degree, are framed.}
\label{F:Normal}
\end{figure}
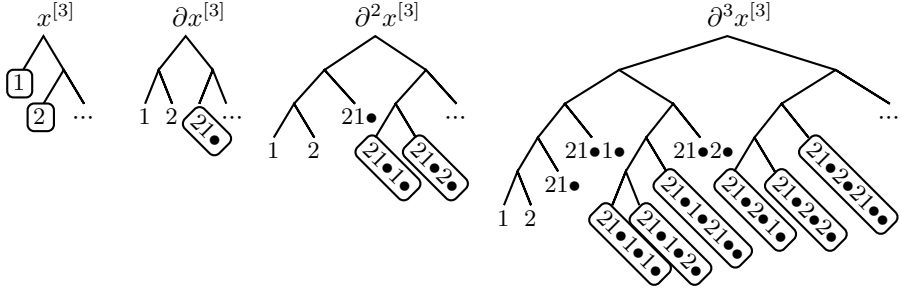

It follows from the proof of Prop.~\ref{P:Normal} (or rather of its counterpart stated in terms of cuts rather than in terms of iterated left subterms) that, for every~$\TT$ in~$\TERM\op\xx$, one can effectively find an LD-expansion of~$\TT$ that is a cut of some term~$\partial^\pp\xx^{[\nn]}$ and, from there, determine the unique normal term that is $\LD$-equivalent to~$\TT$. Therefore, one obtains in this way a new solution for the word problem of~$\eqLD$.

\begin{exam}[\bf normal form solution]\label{X:NFSolution}
Consider $\TT:= (\xx \op \xx) \op (\xx \op (\xx \op \xx))$ and $\TT' := \xx \op ((\xx \op \xx) \op (\xx \op \xx))$ once more. The normal form of~$\TT$ is found by looking for a term~$\partial^\pp\xx^{[\nn]}$ and an address~$\alpha$ such that $\cut(\partial^\pp\xx^{[\nn]})$ is an $\LD$-expansion of~$\TT$: here one easily sees that $\cut(\partial\xx^{[5]}, 011)$ is convenient both for~$\TT$ and for~$\TT'$, and one concludes that the normal form of both~$\TT$ and~$\TT'$ is the normal term~$\xx^{[4]}$.
\end{exam}

\subsection{The case of more than one variable}\label{SS:TwoGener}

So far the word problem for~$\LD$ was solved only in the case of terms involving one variable. Extending the solutions to the case of terms involving any number of variables is easy: essentially, the free left-shelf~$\Free_\nn$ with $\nn \ge 2$ is a lexicographic extension of~$\Free_1$. Here is the key technical point:

\begin{lemm}\label{L:Several}
Two terms $\TT, \TT'$ such that, for some~$\alpha$ and~$\alpha'$, the cuts $\cut(\TT, \alpha)$ and~$\cut(\TT', \alpha')$ coincide except in their rightmost variables, are never $\LD$-equivalent.
\end{lemm}

\begin{proof}
Assume for contradiction $\TT \eqLD \TT'$. Let~$\xx_\ii$ be the rightmost variable in~$\cut(\TT, \alpha)$, and, respectively, let~$\xx_\jj$ be that of~$\cut(\TT', \alpha')$. By the confluence property (Lemma~\ref{L:Confl}), $\TT$ and~$\TT'$ admit a common LD-expansion~$\TT''$. Extending the proof of Lemma~\ref{L:Comp2}, one easily checks that $\TT''$ can be chosen so that some iterated left subterm~$\LS^{\rr}(\TT'')$, is an LD-expansion of~$\cut(\TT, \alpha)$ and some iterated left subterm~$\LS^{\rr'}(\TT'')$ is an LD-expansion of~$\cut(\TT', \alpha')$. By construction, the rightmost variable in~$\LS^\rr(\TT'')$ is~$\xx_\ii$, whereas that of~$\LS^{\rr'}(\TT'')$ is~$\xx_\jj$ with $\jj \not= i$. Hence, we must have $\rr \not= \rr'$, say for instance $\rr < \rr'$. Let~$\TT'''$ be the term obtained from~$\LS^\rr(\TT'')$ by replacing the rightmost variable~$\xx_\ii$ with~$\xx_\jj$. On the one hand, $\LS^{\rr'}(\TT'')$ is a proper iterated left subterm of~$\LS^\rr(\TT'')$, and also of~$\TT'''$, hence $\cut(\TT', \alpha') \divsLD \TT'''$ holds. On the other hand, by construction again, $\TT'''$ is an LD-expansion of the term obtained from~$\cut(\TT, \alpha)$ by replacing the final occurrence of~$\xx_\ii$ with~$\xx_\jj$, which, by assumption, is~$\cut(\TT', \alpha')$, hence $\cut(\TT', \alpha') \eqLD \TT'''$ holds. The conjunction of the above relations contradicts the exclusion property (Lemma~\ref{L:Excl1}).
\end{proof} 

For~$\TT$ in~$\TERM\op\XX$, we denote by~$\overline\TT$ the projection of~$\TT$ where every variable of~$\XX$ is mapped to~$\xx$. Clearly, $\TT \eqLD \TT'$ implies $\overline\TT \eqLD \overline{\TT'}$. The converse implication need not be true, but we have

\begin{lemm}
Assume that $\TT$ and~$\TT'$ are terms in~$\TERM\op\XX$ satisfying $\overline\TT = \overline{\TT'}$. Then $\TT \eqLD \TT'$ holds if, and only if, $\TT$ and~$\TT'$ coincide.
\end{lemm}

\begin{proof}
If $\TT$ and~$\TT'$ do not coincide, there must exist addresses~$\alpha, \alpha'$ as in Lemma~\ref{L:Several}, and the latter then implies that $\TT$ and $\TT'$ are not $\LD$-equivalent.
\end{proof}

We deduce an solution for the word problem that directly extends Algorithm~\ref{A:Syntact}:

\begin{algo}[\bf syntactic solution, general case]\label{A:SyntactBis}

\rm\hfill\par

\noindent {\bf Input}: Two terms $\TT$, $\TT'$ in~$\TERM\op\XX$

\noindent {\bf Output}: $\True$ if $\TT$ and $\TT'$ are $\LD$-equivalent, $\False$ otherwise

\lab{1} $\Found := \False$

\lab{2} $\nn := 0$

\lab{3} {\bf while not} $\Found$ {\bf do}

\lab{4} \quad $\TT_1:= \EXPAND(\TT, \ss_\nn)$

\lab{5} \quad $\TT'_1:= \EXPAND(\TT', \ss'_\nn)$

\lab{6} \quad {\bf if} $\TT_1 = \TT'_1$ {\bf then}

\lab{7} \quad\quad {\bf return} $\True$

\lab{8} \quad\quad $\Found := \True$

\lab{9} \quad {\bf else if} $\overline{\TT_1} = \overline{\TT'_1}$ {\bf or} $\overline{\TT_1} \divs \overline{\TT'_1}$ {\bf or} $\overline{\TT_1} \mults \overline{\TT'_1}$ {\bf then}

\lab{10} \quad\quad {\bf return} $\False$

\lab{11} \quad\quad $\Found := \True$

\lab{12} \quad $\nn:= \nn + 1$

\end{algo}

Extending the semantic algorithm of Section~\ref{SS:Semantic} is also possible. No realization of~$\Free_\nn$ with $\nn \ge 2$ inside Artin's braid group~$B_\infty$ is known (but there is no proof that such a realization cannot exist). However, one can easily define variants of~$B_\infty$ with enough space to build such realizations. A first solution was described by D.\,Larue in~\cite{Lrb}. Another one with a simple geometric interpretation (``charged braids'') appears in~\cite{Dfn}.

\subsection{The Polish algorithm}\label{SS:Polish}

We conclude with one more approach to the word problem of~$\LD$, which leads to a very puzzling open question.

As in Section~\ref{SS:NF}, let us consider the Polish expression of terms. By definition, expanding~$\TT$ to~$\TT'$ using the LD-law means replacing some subterm of~$\TT$ of the form $\TT_1 \op (\TT_2 \op \TT_3)$ with $(\TT_1 \op \TT_2) \op (\TT_1 \op \TT_3)$. On Polish expressions, this means that $\Pol{\TT'}$ is obtained from~$\Pol\TT$ by replacing a factor of~$\Pol\TT$ of the form
$$\Pol{\TT_1}\Pol{\TT_2}\Pol{\TT_3}\opp\opp \text{\qquad with \qquad}
\Pol{\TT_1}\Pol{\TT_2}\opp \Pol{\TT_1} \Pol{\TT_3}\opp\opp:$$
starting from the left, the words~$\Pol\TT$ and~$\Pol{\TT'}$ coincide up to the last letter from~$\Pol{\TT_2}$, which is followed by~$\opp$ in~$\Pol{\TT'}$, whereas, in~$\Pol\TT$, it is followed by the first letter of~$\Pol{\TT_3}$, which is necessarily a variable. This suggests, for comparing two terms~$\TT, \TT'$, to look at the first discrepancy between~$\Pol\TT$ and~$\Pol{\TT'}$, try to expand the term where a variable occurs, and repeat until one possibly obtains twice the same word. This is what we shall call the ``\emph{Polish algorithm}''. To give a precise description, let us review all possible relations between~$\Pol\TT$ and~$\Pol{\TT'}$:

- If $\Pol\TT$ and $\Pol{\TT'}$ coincide, then $\TT$ and~$\TT'$ are equal, hence $\LD$-equivalent.

- If $\Pol\TT$ is a proper prefix of~$\Pol{\TT'}$, or vice versa, then $\TT \divs \TT'$ or $\TT' \divs \TT$ holds, hence, by the exclusion property (Lemma~\ref{L:Excl1}), $\TT$ and~$\TT'$ are not $\LD$-equivalent.

- Otherwise, $\Pol\TT$ and $\Pol{\TT'}$ have a first letter clash. If this clash is of the type ``some~$\xx_\ii$ \vs some~$\xx_\jj$ with~$\ii \not= \jj$'', then, by Lemma~\ref{L:Several}, $\TT$ and~$\TT'$ are not $\LD$-equivalent.

- The remaining case is a first letter clash of the type ``some~$\xx_\ii$ \vs $\opp$''. Then we shall see that there is a unique solution to push the clash further to the right, by expanding the term where~$\xx_\ii$ occurs. Consider a term of the form $\TT_1 \op ((\TT_2 \op \TT_3) \op \TT_4)$: the Polish expression is $\Pol{\TT_1} \Pol{\TT_2} \Pol{\TT_3} \opp \Pol{\TT_4} \opp \opp$. Because $\TT_2$ is nested in~$\TT_2 \op \TT_3$, in order to insert a letter~$\opp$ after~$\Pol{\TT_2}$, we need to first distribute~$\TT_1$ to~$\TT_2 \op \TT_3$, and then distribute~$\TT_1$ to~$\TT_2$, which amounts to successively applying~$\LDop\emptyset$ and~$\LDop0$, obtaining $\Pol{\TT_1} \Pol{\TT_2} \opp \Pol{\TT_1} \Pol{\TT_3} \opp\opp \Pol{\TT_1} \Pol{\TT_4} \opp \opp$, which has the expected form with a~$\opp$ after~$\Pol{\TT_2}$. The situation is generic: 

\begin{defi}[\bf solution]
Let $\alpha$ be an address that contains the factor~$10$. Write~$\alpha$ as~$\beta10^\pp1^\rr$ with $\pp \ge 1$ and $\rr \ge 0$. Then we define the \emph{solution} at~$\alpha$ to be the finite sequence $\SOL(\alpha) := (\beta, \beta0, \beta00 \wdots \beta0^\pp)$. 
\end{defi}

Then $\SOL(\alpha)$ provides a recipe for replacing a variable with the symbol~$\opp$ at a prescribed position in the Polish expansion of a term: 

\begin{lemm}[\bf solution]
The sequence $\SOL(\alpha)$ is the unique finite sequence~$\ss$ with the following property. Assume that $\TT$ is a term in which there is a letter with address~$\alpha$ in~$\Pol{\TT}$ and the next letter in~$\Pol{\TT}$ is a variable. Then the LD-expansion $\TT':= \TT \act \LDop\ss$ is defined,$\Pol{\TT'}$ coincides with~$\Pol\TT$ up to the letter with address~$\alpha$, and the next letter in~$\Pol{\TT'}$ is~$\opp$.
\end{lemm}

The method can then be implemented as follows. For any two terms~$\TT, \TT'$, we define~$\DISC(\TT, \TT')$ to be, when it exists, the pair~$(\pp, \ii)$ in~$\NNNN \times \{1,2\}$ such that $\pp$ is the length of the longest common initial factors between~$\Pol\TT$ and~$\Pol{\TT'}$ and the $(\pp+1)$st letter in~$\TT$ is a variable, whereas the $(\pp+1)$st letter in~$\TT'$ is~$\opp$, in which case we put~$\ii:= 1$, or vice versa, in which case we put~$\ii:= 2$. We recall that $\add(\pp, \TT)$ is the address in~$\TT$ that corresponds to the $\pp$th letter in~$\Pol\TT$ under the correspondence described below Definition~\ref{D:Polish}.

\enlargethispage{5mm}

\begin{algo}[\bf Polish algorithm]\label{A:Polish}

\rm\hfill\par

\noindent {\bf Input}: Two terms $\TT$, $\TT'$ in~$\TERM\op\XX$

\noindent {\bf Output}: $\True$ if $\TT$ and $\TT'$ are $\LD$-equivalent, $\False$ otherwise

\lab{1} {\bf while} $\DISC(\TT, \TT')$ is defined {\bf do}

\lab{2} \quad $(\pp, \ii):= \DISC(\TT, \TT')$

\lab{3} \quad {\bf if} $\ii = 1$ {\bf then}

\lab{4} \quad\quad $\alpha:= \add(\pp, \TT)$

\lab{5} \quad\quad $\TT:= \EXPAND(\TT, \SOL(\alpha))$

\lab{3} \quad {\bf if} $\ii = 2$ {\bf then}

\lab{4} \quad\quad $\alpha:= \add(\pp, \TT')$

\lab{5} \quad\quad $\TT':= \EXPAND(\TT', \SOL(\alpha))$

\lab{6} {\bf if} $\TT = \TT'$ {\bf then}

\lab{7} \quad {\bf return} $\True$

\lab{8} {\bf if} $\TT \divs \TT'$ {\bf or} $\TT' \divs \TT$ {\bf then}

\lab{9} \quad {\bf return} $\False$

\end{algo}

\begin{exam}[\bf Polish algorithm]\label{X:Polish}
Consider $\TT:= (\xx_1 \op \xx_2) \op (\xx_1 \op (\xx_3 \op \xx_4))$ and $\TT' := \xx_1 \op ((\xx_2 \op \xx_3) \op (\xx_2 \op \xx_4))$, a multi-variable version of the terms of Examples~\ref{X:Syntactic} and~\ref{X:Semantic}. Then we start with

\smallskip $\Pol{\TT_0} = \xx_1\xx_2 \,\Vert\, \opp \xx_1\xx_3\xx_4\opp \opp \opp$,

$\Pol{\TT'_0} = \xx_1\xx_2 \,\Vert\, \xx_3\opp \xx_2\xx_4\opp \opp \opp$

\noindent $($we use the symbol~$\Vert$ to emphasize the longest common prefix$)$.

\smallskip\noindent The words~$\Pol{\TT_0}$ and $\Pol{\TT'_0}$ have a clash after the second letter, followed by~$\opp$ in~$\Pol{\TT_0}$ and by~$\xx_3$ in~$\Pol{\TT'_0}$. Thus $\DISC(\TT_0, \TT'_0)$ is defined, and equal to~$(2, 2)$. We then find $\add(2, \TT'_0) = 100$, $\SOL(100) = (\emptyset, 0)$, and LD-expanding~$\TT'_0$ at~$\emptyset$ and then at~$0$ yields

\smallskip $\Pol{\TT_1} = \xx_1\xx_2\opp \xx_1\xx_3 \,\Vert\, \xx_4 \opp \opp \opp$,
 
$\Pol{\TT'_1} = \xx_1\xx_2\opp \xx_1\xx_3 \,\Vert\, \opp \opp \xx_1\xx_2\xx_4\opp \opp \opp$.
 
\smallskip\noindent Now $\Pol{\TT_1}$ and $\Pol{\TT'_1}$ have a clash after the fifth letter, followed by~$\xx_3$ in~$\Pol{\TT_1}$ and by~$\opp$ in~$\Pol{\TT'_1}$. Thus $\DISC(\TT_1, \TT'_1)$ is defined, and equal to~$(5, 1)$. We find $\add(5, \TT_1) = 110$, $\SOL(110) = (1)$, and expanding~$\TT_1$ at~$1$ yields

\smallskip $\Pol{\TT_2} = \xx_1\xx_2\opp \xx_1\xx_3\opp \,\Vert\, \xx_1\xx_4 \opp \opp \opp$,
 
$\Pol{\TT'_2} = \xx_1\xx_2\opp \xx_1\xx_3\opp \,\Vert\, \opp \xx_1\xx_2\xx_4\opp \opp \opp$.

\smallskip\noindent Then $\Pol{\TT_2}$ and $\Pol{\TT'_2}$ have a clash after the sixth letter, followed by~$\xx_1$ in~$\Pol{\TT_2}$ and by~$\opp$ in~$\Pol{\TT'_2}$. Thus $\DISC(\TT_2, \TT'_2)$ is defined, and equal to~$(6, 1)$. We find $\add(6, \TT_2) = 10$, $\SOL(10) = (\emptyset)$, and expanding~$\TT_2$ at~$\emptyset$ yields

\smallskip
$\Pol{\TT_3} = \xx_1\xx_2\opp \xx_1\xx_3\opp \opp \xx_1\xx_2\,\Vert\,\opp \xx_1\xx_4\opp \opp \opp$,

$\Pol{\TT'_3} = \xx_1\xx_2\opp \xx_1\xx_3\opp \opp \xx_1\xx_2\,\Vert\, \xx_4\opp \opp \opp$.

\smallskip\noindent The words~$\Pol{\TT_3}$ and $\Pol{\TT'_3}$ have a clash after the ninth letter, followed by~$\opp$ in~$\Pol{\TT_3}$ and by~$\xx_4$ in~$\Pol{\TT'_3}$. Thus $\DISC(\TT_3, \TT'_3)$ is defined, and equal to~$(10, 2)$. We find $\add(10, \TT'_3) =\nobreak 110$, $\SOL(110) = (1)$, and expanding~$\TT'_3$ at~$1$ yields

\smallskip
$\Pol{\TT_4} = \xx_1\xx_2\opp \xx_1\xx_3\opp \opp \xx_1\xx_2\opp \xx_1\xx_4\opp \opp \opp$,

$\Pol{\TT'_4} = \xx_1\xx_2\opp \xx_1\xx_3\opp \opp \xx_1\xx_2\opp \xx_1\xx_4\opp \opp \opp$.

\smallskip\noindent The words~$\Pol{\TT_4}$ and $\Pol{\TT'_4}$ are equal: the Polish Algorithm running on the pair~$(\TT, \TT')$ converges to~$(\TT_4, \TT_4)$ in 4 steps. We conclude that $\TT$ and~$\TT'$ are $\LD$-equivalent. 
\end{exam}

As in the case of Algorithms~\ref{A:Syntact} and~\ref{A:SyntactBis}, the correctness of the Polish algorithm follows from the exclusion property (Lemma~\ref{L:Excl1}): if the algorithm converges to a pair of equal terms~$(\TT_\nn, \TT'_\nn)$, the term~$\TT_\nn$ is a common LD-expansion of the initial terms, and the latter are $\LD$-equivalent; if the algorithm converges to a pair of terms~$(\TT_\nn, \TT'_\nn)$ with $\TT_\nn \divs \TT'_\nn$ or $\TT'_\nn \divs \TT_\nn$, then, by the exclusion property, $\TT_\nn$ and $\TT'_\nn$ cannot be $\LD$-equivalent and, therefore, the initial terms cannot either. 

But this leaves the following question open:

\begin{ques}[\bf convergence]
Does the Polish algorithm always converge?
\end{ques}

The problem seems very difficult---and, with the Embedding Conjecture (Question~\ref{Q:Emb}), it is the most puzzling open problem involving syntactic aspects of selfdistributivity. Experimental results and a number of partial results~\cite{Dfl, Dgf} suggest a positive answer, but no complete result is known so far. Due to the nature of selfdistributivity, the sizes of LD-expansions quickly increase (see Fig.~\ref{F:Normal}). However, many patterns are repeated in such LD-expansions, and clever encodings can lower the size, on the shape of S.\,Schleimer's approach in~\cite{Schleimer}. By doing so, O.\,Deiser~\cite{Dei} could perform exhaustive search for all pairs of terms up to size~$9$: interestingly, he discovered that the Polish algorithm usually converges very fast, except in a few isolated cases, where it still converges, but in an extremely long time. Let us mention that other algebraic laws are eligible for the Polish algorithm: as can be expected, convergence in the case of associativity is trivial, whereas selfdistributivity seems to be the most difficult one---but many questions remain open~\cite{Dei}.

\section{Word problem, the case of racks, quandles, and spindles}\label{S:WPOther}

We conclude the survey with a similar investigation of the word problem for the three derived notions obtained by adding additional laws to selfdistributivity, namely racks, quandles, and spindles. We shall successively review the case of free racks and free quandles, which are very similar and easy (Subsection~\ref{SS:WPQuandle}), and finish with the case of free spindles, which seems very difficult (Subsection~\ref{SS:WPSpindle}).

\subsection{The case of racks and quandles}\label{SS:WPQuandle}

The definition of racks and quandles comes in two versions, according to whether one or two binary operations are considered. If only one operation is concerned, the condition that the right translations are bijective does not correspond to obeying an algebraic law, and there is no natural word problem. By contrast, when two operations~$\opR, \opRb$ are considered, being a rack (\resp a quandle) corresponds to obeying~\eqref{RD} plus the two algebraic laws of~\eqref{E:InvOp} (\resp these, plus the idempotency law $\xx \opR \xx = \xx$), and the word problem is a well posed question. We denote by~$\eqRack$ and $\eqQuandle$ the congruences on $\TERM{\opR, \opRb}\XX$ generated by the instances of the laws~RD and~\eqref{E:InvOp} (\resp these plus idempotency). By construction, the quotient-structure~$\TERM{\opR,\opRb}\XX{/}{\eqRack}$ is a free rack based on~$\XX$, and, similarly, $\TERM{\opR,\opRb}\XX{/}{\eqQuandle}$ is a free quandle based on~$\XX$.

We shall see that the word problems are easy here, because there exists a simple family of distinguished (or ``normal'') terms. The general principle is as follows:

\begin{lemm}\label{L:Normal}
Let $\LLL$ be a family of algebraic laws involving a signature~$\Sigma$. Let $\NT$ be a subset of~$\TERM\Sigma\XX$. Let $\SSS$ be a $\Sigma$-structure generated by~$\XX$. Assume that

\ITEM1 every term in~$\TERM\Sigma\XX$ is $\eqL$-equivalent to at least one term in~$\NT$,

\ITEM2 distinct terms in~$\NT$ have distinct evaluations in~$\SSS$.

\noindent Then $\SSS$ is a free $\LLL$-algebra based on~$\XX$.
\end{lemm}

\begin{proof}
Write $\eval(\TT, \SSS)$ for the evaluation of a term~$\TT$ in~$\SSS$. Let $\TT, \TT'$ be two terms in~$\TERM\XX\Sigma$. By~\ITEM1, there exist $\TT_1$ and~$\TT'_1$ in~$\NT$ satisfying $\TT \eqL\nobreak \TT_1$ and~$\TT' \eqL\nobreak \TT'_1$. As $\eqL$ is transitive, $\TT \eqL\nobreak \TT'$ is equivalent to $\TT_1 \eqL\nobreak \TT'_1$, hence, by~\ITEM2, to $\eval(\TT_1, \SSS) = \eval(\TT'_1, \SSS)$. As $\SSS$ is an $\LLL$-algebra, $\TT \eqL \TT_1$ implies $\eval(\TT, \SSS) =\nobreak \eval(\TT_1, \SSS)$ and, similarly, $\eval(\TT', \SSS) =\nobreak \eval(\TT'_1, \SSS)$. Finally, we deduce that $\TT \eqL \TT'$ is equivalent to~$\eval(\TT, \SSS) =\nobreak \eval(\TT', \SSS)$. Hence, $\SSS$ is free.
\end{proof}

Applying this to the cases of racks and quandles is easy.

\begin{prop}[\bf realization]\label{P:FreeConj}
Let~$\FG\XX$ be a free group based on a set~$\XX$. 

\ITEM1 The structure~$\HalfConj(\XX,\FG\XX)$ is a free rack based on~$\XX$.

\ITEM2 The structure~$\Conjjj(\FG\XX)$ is a free quandle based on~$\XX$.
\end{prop}

\begin{proof}[Proof (principle)]
Write~$\opR^{+1}$ for~$\opR$ and $\opR^{-1}$ for~$\opRb$, and define~$\NT$ to be the family of all terms of the particular form
\begin{equation}\label{E:Normal1}
( \pdots ((\xx \opR^{\ee_1} \xx_1) \opR^{\ee_2} \xx_2) \pdots ) \opR^{\ee_\nn} \xx_\nn,
\end{equation}
with $\nn \ge 0$, $\xx, \xx_1 \wdots \xx_\nn \in \XX$, $\ee_1 \wdots \ee_\nn = \pm1$, and $\ee_\ii \not= \ee_{\ii + 1}$ implying $\xx_\ii \not= \xx_{\ii + 1}$. An easy induction shows that every term in~$\TERM{\opR,\opRb}\XX$ is $\eqRack$-equivalent to a term as in~\eqref{E:Normal1}. Next, the evaluation of such a term in~$\HalfConj(\XX, \FG\XX)$ is the pair 
\begin{equation*}\label{E:Normal2}
(\xx,  \xx_1^{\ee_1} \xx_2^{\ee_2} \pdots \xx_\nn^{\ee_\nn}),
\end{equation*}
where the second entry is a freely reduced signed $\XX$-word. Hence distinct terms of the form~\eqref{E:Normal1} have distinct $\HalfConj(\XX, \FG\XX)$-evaluations. By Lemma~\ref{L:Normal}, the structure $\HalfConj(\XX, \FG\XX)$ is a free rack based on~$\XX$.

The argument is similar for quandles, demanding in addition $\xx_1 \not= \xx$ in~\eqref{E:Normal1}. The evaluation in~$\Conjjj(\FG\XX)$ is then the freely reduced signed $\XX$-word
\begin{equation*}\label{E:Normal2}
\xx_\nn^{-\ee_\nn} \pdots \xx_2^{-\ee_2} \xx_1^{-\ee_1} \xx \xx_1^{\ee_1} \xx_2^{\ee_2} \pdots \xx_\nn^{\ee_\nn},
\end{equation*}
which again determines the term~\eqref{E:Normal1} it comes from.
\end{proof}

We derive a semantic algorithm for the word problems of racks and quandles. Below we use $\Red$ for free group reduction, that is, iteratively deleting length~$2$ factors of the form~$\xx\xx\inv$ or~$\xx\inv\xx$. We use $\ew$ for the empty word.

\begin{algo}[\bf semantic algorithm for~$\eqRack$]\label{A:Rack}

\rm\hfill\par

\noindent {\bf Input}: Two $\XX$-terms $\TT$, $\TT'$

\noindent {\bf Output}: $\True$ if $\TT$ and $\TT'$ are $\eqRack$-equivalent, $\False$ otherwise

\lab{1} $(\xx, \ww):= \EVAL(\TT)$

\lab{2} $(\xx', \ww'):= \EVAL(\TT')$

\lab{3} {\bf if} $\xx \not= \xx'$ {\bf or} $\Red(\ww) \not= \Red(\ww') $ {\bf then}

\lab{4} \quad {\bf return} $\False$

\lab{5} {\bf else}

\lab{6} \quad {\bf return} $\True$

\medskip

\lab{7} {\bf function} $\EVAL(\TT:$ term): element of 
$\XX \times (\XX \cup \XX\inv)^*$

\lab{8} {\bf if} $\TT = \xx \in \XX$ {\bf then}

\lab{9} \quad {\bf return} $(\xx, \ew)$

\lab{10} {\bf else if} $\TT = \TT_1 \opR \TT_2$ {\bf then}

\lab{11} \quad {\bf return} $\EVAL(\TT_1) \opR \EVAL(\TT_2)$ with $\opR$ as in~\eqref{E:Half}

\lab{12} {\bf else if} $\TT = \TT_1 \opRb \TT_2$ {\bf then}

\lab{13} \quad {\bf return} $\EVAL(\TT_1) \opRb \EVAL(\TT_2)$ with $\opRb$ as in~\eqref{E:HalfBis}

\end{algo}

Composing free reduction with the function~$\EVAL$ provides the evaluation in the structure~$\HalfConj(\XX,\FG\XX)$. So, Algorithm~\ref{A:Rack} is a direct translation of Prop.~\ref{P:FreeConj}\ITEM1.

An entirely similar algorithm solves the word problem for the quandle laws, at the expense of replacing the evaluation in~$\HalfConj(\XX, \FG\XX)$ with one in~$\Conjjj(\FG\XX)$ using the conjugacy operations of~\eqref{E:Conj}, and appealing to Prop.~\ref{P:FreeConj}\ITEM2.

\subsection{The case of spindles}\label{SS:WPSpindle}

Let us finally address the word problem for free spindles, that is, the question of deciding whether two terms~$\TT, \TT'$ are LDI-equivalent, where LDI refers to the conjunction of (left) selfdistributivity~LD and idempotency~I. We use~$\eqLDI$ for the associated congruence on terms.

The word problem of~$\LDI$ is trivial for terms involving only one variable, as an easy induction gives $\TT \eqLDI \xx$ for every~$\TT$ in~$\TERM\op\xx$. As a consequence, the free (left) spindle on one generator has only one element.

Frustratingly, this case is the only one, for which a solution is known. No semantic solution is known, because no realization of free (left) spindles is known. As observed in Example~\ref{X:Conj}, the conjugacy structure~$\Conjj\GG$ of any group~$\GG$ is a quandle, hence a spindle. One might think that, starting with a free group, the associated conjugacy spindle might be free. This is not true:

\begin{prop}[\bf not free]
If $\GG$ is a group distinct of~$\{1\}$, the conjugacy spindle~$\Conjj\GG$ is not free.
\end{prop}

\begin{proof}[Principle of proof]
There exist algebraic laws satisfied by conjugation and not consequences of~$\LDI$~\cite{DKM}, a typical example being
\begin{equation}\label{E:ConjL}
((\xx \op \yy) \op \yy) \op (\yy \op \zz) = (\xx \op \yy) \op ((\yy \op \xx) \op \zz).
\end{equation}

\noindent\begin{minipage}{\textwidth}\rightskip28mm
Then \eqref{E:ConjL} holds in every conjugacy left-spindle since, when $\xx \op \yy$ is evaluated to~$\xx \yy \xx\inv$, the two terms of~\eqref{E:ConjL} evaluate to~$\xx \yy \xx\inv \yy \xx \yy\inv \zz \yy \xx\inv \yy\inv \xx \yy\inv \xx\inv$, whereas \eqref{E:ConjL} fails for $\xx = \zz = 2$ and $\yy = 3$ in the four-element left-spindle whose table is shown on the right.\VR(0,2)
\hfill\begin{picture}(0,0)(-2.5,-9.5)
\put(0,0){\footnotesize\begin{tabular}{c|cccc}
1&1&2&3&4\\ \hline 1&1&2&3&4\\ 2&1&2&1&2\\ 3&3&4&3&4\\ 4&1&2&3&4
\end{tabular}}
\end{picture}
\end{minipage} 

An infinite family of similar laws is constructed in~\cite{Lrc}, and it is shown that no finite subfamily generates it. Also see~\cite{Sta} and~\cite{Dfy}.
\end{proof}

In the direction of syntactic solutions, a counterpart of the approach of Section~\ref{SS:Thompson} was successfully developed by P.\,Jedli\v cka in~\cite{Jed1, Jed2, Jed3}: in the case of LDI, the confluence property still holds, and one can investigate a geometry monoid~$\Geom\LDI$ similar to~$\Geom\LD$ but, at least because no relevant blueprint was found so far, the approach falls short of solving the word problem. Thus, once again, we meet with a puzzling open question:

\begin{ques}
Is the word problem for~$\LDI$ solvable?
\end{ques}

A positive answer may seem likely, but it remains unknown so far.


\end{document}